\documentclass[a4paper,12pt]{article}
\def\ff{\varphi}

\usepackage{hyperref}

\usepackage[english]{babel}
\usepackage{amsthm,amsmath,amssymb}

\newtheorem{rema}{Remark}

\newtheorem{claim}{Claim}
\newtheorem{lemma}{Lemma}
\newtheorem{corollary}{Corollary}
\newtheorem{prop}{Proposition}
\newtheorem{thm}{Theorem}

\def\imp{\underset{\text{a.s.}}{\Longrightarrow}}

\newcommand{\rmd}{ \mathrm{d}}

\def \d   {\delta}

\def \eps {\varepsilon}

\def\E{{\mathbb{E}}}

\def\P{{\mathbb{P}}}
\def\R{{\mathbb{R}}}
\def\N{{\mathbb{N}}}

\def\F{{\cal{F}}}
\def\XX{{\cal X}}

\newcommand{\supp}{\mathop{\mathrm{supp}}}

\def\limt{\lim_{t\to\infty}}
\def\limt0{\lim_{t\to 0}}

\def\|{ | }
\def\nn{{\sf n}}

\newcommand\as{\overset{\text{a.s.}}{\longrightarrow}}

\title{Convergence in the $p$-contest}
\author{Philip Kennerberg${}^*$ and Stanislav Volkov\footnote{Centre for Mathematical Sciences, Lund University, Box 118 SE-22100, Lund, Sweden, The research is partially supported  by the Swedish Research Council grant VR~2014-5157 and by Crafoord foundation grant 20190667.}} 

\begin{document}
\maketitle

\begin{abstract}
We study asymptotic properties of the following Markov system of $N \geq 3$ points in~$[0,1]$. At each time step, the point farthest from the current centre of mass, multiplied by a constant $p>0$, is removed and replaced by an independent $\zeta$-distributed point; the problem, inspired by variants of the Bak--Sneppen model of evolution and called {\it a $p$-contest}, was posed in~\cite{gkw}.  We obtain various criteria for the convergences of the system, both for $p<1$ and $p>1$.

In particular, when $p<1$ and $\zeta\sim U[0,1]$, we show that the  limiting configuration converges to zero. When $p>1$, we show that the configuration must converge to either zero or one, and we present an example where both outcomes are possible. Finally, when $p>1$, $N=3$ and $\zeta$ satisfies certain mild conditions (e.g.~$\zeta\sim U[0,1]$), we prove that the configuration converges to one a.s.

Our paper substantially extends the results of~\cite{GVW,KV} where it was assumed that $p=1$. Unlike the previous models, one can no longer use the Lyapunov function based just on the radius of gyration; when $0<p<1$ one has to find a more finely tuned function which turns out to be a supermartingale; the proof of this fact constitutes an unwieldy, albeit necessary, part of the paper. 
\end{abstract}



\section{Introduction}
This paper extends the results of~\cite{GVW} and~\cite{KV}
on the so-called {\em Keynesian beauty contest}, or, as it was called in~\cite{KV}, {\em Jante's law process}. Following~\cite{GVW}, we recall that in the Keynesian beauty contest, we have $N$ players guessing a number, and the person who guesses closest to the mean of all the $N$ guesses wins; see~\cite[Ch.\ 12, \S V]{keynes}.  The formal version, suggested by Moulin \cite[p.\ 72]{moulin}, assumes that this game is played by choosing numbers on the interval $[0,1]$, the  ``$p$-beauty contest'', in which the target is the mean value, multiplied  by a constant $p>0$.  For the applications of the $p$-contest in the game theory, we refer the reader to e.g.~\cite{ST}; see also~\cite{degr} and~\cite{GVW} and references therein for further applications and other relevant papers.

The version of the $p$-contest with $p\equiv 1$ was studied in~\cite{GVW,KV}. In~\cite{GVW} it was shown that in the model where at each unit of time the point farthest from the center of mass is replaced by a point chosen uniformly on $[0,1]$, then eventually all (but one) points converge almost surely to some random limit the support of which is the whole interval $[0,1]$; many of the results were extended for the version of the model on $\R^d$, $d\ge 2$. The results of~\cite{GVW} were further generalized  in~\cite{KV}, by removing the assumption that a new point is chosen uniformly on $[0,1]$, as well as by removing more than one point at once, these points being chosen in such a way that the moment of inertia of the resulting configuration is minimized. However, the case $p\ne 1$ was not addressed in either of these two papers.

Let us now formally define the model; the notation will be similar to those in~\cite{GVW,KV}. Let $\XX = \{ x_1, x_2, \ldots, x_N\}\in\R^N$ be an unordered $N$-tuple of points in $\R$, and $ (x_{(1)} , x_{(2)} , \dots, x_{(N)})$ be these points put in non-decreasing order, that is, $x_{(1)}\le x_{(2)}\le\dots\le x_{(N)}$. 
As in~\cite{GVW,KV} let us define the barycentre of the configuration as
\begin{align*}
\mu_N (x_1,\dots,x_N)&:= N^{-1} \sum_{i=1}^N x_i.
\end{align*}
Fix some $p>0$ and also define the {\em $p-$centre of mass} as $p\mu_N(x_1,\dots,x_N)$.

The point, farthest from the $p-$centre of mass, is called the {\em extreme} point of~$\XX$, and it can be either $x_{(1)}$ or $x_{(N)}$ (with possibility of a tie), and {\em the core} of $\XX$, denoted by $\XX'$, is constructed from $\XX$ by removing  the extreme point; in case of a tie between the left-most and the right-most point, we choose either of them with equal probability (same as in~\cite{GVW,KV}). Throughout the rest of the paper, $x_{(1)}(t),\dots,x_{(N-1)}(t)$ shall denote the points of the core\footnote{rather than of $\XX(t)$} $\XX'(t)$ put into non-decreasing order.

Our process runs as follows. Let $\XX(t)=\{X_1(t), \dots, X_N(t)\}$ be an unordered $N$-tuple of points in~$\R$ at time $t=0,1,2,\dots$. Given~$\XX(t)$, let~$\XX'(t)$ be the core of~$\XX(t)$ and replace $\XX(t)\setminus \XX'(t) $ by a $\zeta$-distributed random variable so that 
$$
\XX (t+1) = \XX'(t) \cup \{\zeta_{t+1}\},
$$
where $\zeta_{t}$, $t=1,2,\dots$, are i.i.d.\ random variables with a common distribution~$\zeta$.

Finally, to finish the specification of our process, we allow the initial configuration $\XX(0)$ to be arbitrary or random, with the only requirement being that all the points of $\XX(0)$ must lie in the support of $\zeta$.

Throughout the paper we will use the notation $A\imp B$ for two events $A$ and $B$, whenever $\P(A\cap B^c)=0$, that is, when $A\subseteq B$ up to a set of measure $0$.
We will also write, with some abuse of notations, that $\lim_{t\to\infty}\XX'(t)= a \in \R$ or equivalently $\XX'(t)\to a$ as $t\to\infty$ if $\XX'(t)\to (a,a,\dots,a)\in\R^{N-1}$, i.e.\ $\lim_{t\to\infty}x_{(i)}(t)=a$ for all $i=1,2,\dots,N-1$. 
Similarly, for an  interval $(a,b)$ we will write 
$\XX'(t)\in (a,b)$ whenever {\em all} $x_{(1)}(t),\dots,x_{(N-1)}(t)\in (a,b)$.
Finally, we will assume that $\inf\emptyset=+\infty$, and use the notation $y^+=\max(y,0)$ for $y\in \R$.

Also we require that $\zeta$ has a {\it full support} on $[0,1]$, that is, $\P(\zeta\in(a,b))>0$ for all $a,b$ such that  $0\le a<b\le 1$.

\section{The case $p<1$}
Throughout this Section we assume that $0<p<1$ and that $\supp\zeta=[0,1]$.  Because of the scaling invariance, our results may be trivially extended to the case when $\supp\zeta=[0,A]$, $A\in(0,\infty)$; some of them are even true when $A=\infty$; however, to simplify the presentation from now on we will deal only with the case $A=1$. 

First, we present some general statements; more precise results will follow in case where $\zeta\sim U[0,1]$.

\begin{prop}\label{liminfzero}
We have
\begin{itemize}
\item[(a)]
$\liminf_{t\to\infty}x_{(N-1)}(t)=0$;
\item[(b)]
$\P\left(\exists \lim_{t\to\infty} \XX'(t)\in(0,1]\right)=0$;
\item[(c)]
if $p<\frac 12 + \frac{1}{2(N-1)}$ then $\P\left(\lim_{t\to\infty} \XX'(t)=0\right)=1$;
\item[(d)]
if $p< \frac 12 + \frac{1}{N-2}$ then $\{x_{(1)}(t)\to 0\}\imp \{\lim_{t\to\infty}\XX'(t)= 0\}$.
\end{itemize}
\end{prop}

\begin{proof}
(a) Since $\zeta$ has full support on $[0,1]$ it follows that (see~\cite{KV}, Proposition~1) there exists a function $f:\R^+\to\R^+$ such that 
\begin{align}\label{eqfullsupp}
\P(\zeta\in (a,b))\ge f(b-a)>0\qquad\text{for all }0\le a<b\le 1.
\end{align}
Also, to simplify notations, we write $\mu=\mu_N(\XX(t))$ throughout the proof.

Fix a small positive $\eps$ such that $p+2\eps<1$. Suppose that for some $t$ we have $x_{(N-1)}(t)\le b\le 1$. We will show that $x_{(N-1)}(t+N)\le b(1-\eps)$ with a strictly positive probability which only depends on $p,b,\epsilon$ and $N$.
Assume that we have $\zeta_{t+1},\dots,\zeta_{t+N-1}\in\left(pb,(p+\eps)b \right)\subset (pb,b)$; this happens with probability no less than $\left[f(p\eps b)\right]^{N-1}$. We claim that by the time $t+N$ we have $x_{(N-1)}(t+N-1)<(p+\eps)b$. Indeed, $p\mu\le pb$ always lies to the left of the newly sampled points, therefore either there are no more points to the right of $(p+\eps)b$ at some time $s\in [t,t+N-1]$ (which implies that there will be no points there at time $t+N$ due to the sampling range of the new points), or one of the older points, i.e.\ present at time $t$, gets removed (it can be the one to the left of~$pb$). Since we eventually have to replace all the $N-1$ old points, then $x_{(N-1)}(t+N)\le b(1-\eps)$.

Fix a $\delta>0$ and find $M$ so large that $(1-\eps)^M<\delta$. Let the event $C(s)=\{x_{(N-1)}(s) <\delta\}$. By iterating the above argument, we get that $\P(C(t+NM)\| \F_t) \ge \prod_{i=1}^M \left[f(p\eps(1-\eps)^{i-1})\right]^{N-1}$, since at time $t$ we can set $b=1$. Therefore, $\sum_m \P(C(NM(m+1))\| \F_{NMm})=\infty$ and by L\'evy's extension of the Borel-Cantelli lemma (see e.g.~\cite{DW}) infinitely many $C(s)$ occur. Since $\delta>0$ is arbitrary, we get $\liminf_{t\to\infty} x_{(N-1)}(t)=0$. 
\\[5mm]
(b)
Let $r=\frac {1+p^{-1}}{2}>1$. Suppose that the core converges to some point $x\in (0,1]$; then there exist a rational $q\in (0,1]$ and a $T>0$ such that $\XX'(t)\in(q,rq)$ for all $t\ge T$, formally
\begin{align}\label{eqqq}
\{\exists \lim \XX'(t)\in (0,1]\} \subseteq \bigcup_{q\in Q\cap (0,1]} \bigcup_{T>0} \bigcap_{t\ge T} A_{q,t}
\end{align}
where $A_{q,t}=\{\XX'(t)\in(q,rq)\}$.
We will show that
$$
\P(A_{q,t+1}\| \F_t,A_{q,t})<1-\nu_q\quad\text{for all }t
$$
for some $\nu_q>0$. This will imply, in turn, that
$$
\P\left(\bigcap_{t\ge T} A_{q,t}\right)=0
$$
and hence the RHS (and thus the LHS as well) of~\eqref{eqqq} has the probability $0$.

Suppose $A_{q,t}$ has occurred and the newly sampled point $\zeta\in(pq,q)$. Then
\begin{align*}
p\mu_N(\XX'(\tau_k)\cup\{\zeta\})
&< p rq=\frac{pq+q}{2}<\frac{\zeta+x_{(N-1)}}{2} 
\end{align*}
Consequently, $x_{(N-1)}$ lies further from the $p-$center of mass, and hence it should be removed. The new configuration will, however, contain the point $\zeta\notin(q,rq)$ and hence $A_{q,t+1}$ does not occur.
Thus 
$$
\P\left( A_{q,t+1} \| \F_t,A_{q,t} \right)\le 
1-\P\left( \zeta\in(q,rq) \right)\le 1- f(pq-q)=:1-\nu_q
$$
as required.
\\[5mm]
(c)
First, we will show that it is the right-most point of the configuration which should be always removed; note that it suffices to check this only when $x_{(N)}>0$.
Indeed, by the assumption on~$p$ we have
$$
\mu\le  \frac{(N-1)x_{(1)}+(N-1)x_{(N)}}{N}=\frac{2p(N-1)}{N}
\cdot\frac{x_{(1)}+x_{(N)}}{2p}
<\frac{x_{(1)}+x_{(N)}}{2p}
$$
implying that
$$
x_{(N)}-p\mu>p\mu-x_{(1)}
\Longleftrightarrow
x_{(N)}-p\mu>|p\mu-x_{(1)}|
$$
Therefore, $x_{(N)}$ is the farthest point from the $p-$centre of mass. This implies that $x_{(N-1)}(t)$ is non-increasing and therefore result now easily follows from part (a) since $x_{(N-1)}(t)$ is an upper bound for all the core points.
\\[5mm]
(d)
Apply Corollary~\ref{corintN2p}  with $k=1$; it is possible because of Remark~\ref{extraplessone}.
\end{proof}

We are ready to present the main result of this Section.
\begin{thm}\label{thmplessone}
Suppose that $\zeta\sim U[0,1]$. Then $\XX'(t)\to 0$ a.s.
\end{thm}

\begin{proof}
Proposition~\ref{liminfzero} (c) implies that we now only need to consider the case $p\ge \frac{N}{2(N-1)}$, which we will assume from now on.

Let us introduce a modification of this process on $[0,+\infty)$ which we will call the {\it borderless $p$-contest}; it is essentially the same process as the one in Section~3.4 of~\cite{GVW}.  In order to do this, we need the following statement.

\begin{lemma}\label{lemborderless}
Suppose that $x_1,\dots,x_{N-1}> 0$. Then there exists an $R=R(x_{(N-1)})\ge 0$ such that $x$ is the farthest point from $p\mu=\frac{p}{N}(x_1+\dots+x_{N-1}+x)$ whenever $x>R$.
\end{lemma}
\begin{proof}[Proof of Lemma~\ref{lemborderless}.]
Set $R=6 x_{(N-1)}$. Then $x>x_{(1)}$ is farther from the centre of mass than $x_{(1)}$ if and only if
$$
x-p\mu>|p\mu-x_{(1)}|
\Longleftrightarrow
x-p\mu>p\mu-x_{(1)}
\Longleftrightarrow
x\left(1-\frac{2p}N\right)>2p\frac{x_1+\dots+x_{N-1}}{N}-x_{(1)}
$$
This is true, due to the fact that $x>R$ and 
$$
x\left(1-\frac{2p}N\right)>\frac{x}3>2x_{(N-1)}>2p x_{(N-1)}
>2p\frac{x_1+\dots+x_{N-1}}{N}
$$
since $p<1$ and $N\ge 3$.
\end{proof}

The borderless process is constructed as follows. Our core configuration starts as before in $[0,1]$, and we use the same rejection/acceptance criteria for new points. However, we will now allow points to be generated to the right of $1$ as well. Let $R_t=R(x_{(N-1)}(t))$ where $R$ is taken from Lemma~\ref{lemborderless}. Then a new point is sampled uniformly and independently of the past on the interval $[0,R_t]$; formally, it is given by $R_t U_t$ where $U_t$ are i.i.d.\ uniform $[0,1]$ random variables independent of everything. Observe that if we consider the embedded process only at the times when the core configuration changes, then the exact form of the function $R(\cdot)$ is irrelevant, due to the fact that the uniform distribution conditioned on a subinterval is also uniform on that subinterval.

Next, for $y=\{y_1,\dots,y_{N-1}\}$ define the function
\begin{align}\label{Lyap}
h(y)=F(y)+k\mu(y)^2,
\end{align}
where 
$$
F(y)=\sum_{i=1}^{N-1} (y_i-\mu(y))^2,
\quad \mu(y)=\frac{1}{N-1} \sum_{i=1}^{N-1} y_i,
\quad k=\frac{(N-1)^2(1-p)}{N-2}.
$$

We continue with the following
\begin{lemma}\label{lemMod}
For the borderless $p$-contest the sequence of random variables $h\left(\XX'(t)\right)\ge 0$, $t=1,2,\dots$, is a supermartingale.
\end{lemma}
\begin{rema}
Note that the function $F(\cdot)$ defined above is a Lyapunov function for the process in~\cite{GVW}; this is no longer the case as long as $p\ne 1$; that is why we have to use a carefully chosen ``correction" factor which involves the barycentre of the configuration.
\end{rema}

\begin{proof}[Proof of Lemma~\ref{lemMod}]
Assume that $x_{(N-1)}(t)>0$ (otherwise the process has already stopped, and the result is trivial). The inequality, which we want to obtain
is 
$$
\left. \E[h(\XX'(t+1))-h(\XX'(t))\| \F_t] \right|_{x(t)=y}\le 0
$$
for all $y=(y_1,\dots,y_{N-1})$ with $y_i \in [0,1]$.
Note that the function $h(y)$ is homogeneous of degree~$2$ in~$y$, therefore w.l.o.g.\ we can assume that $\max y\equiv 1$.

For simplicity let $M=N-1\ge 2$, and let
$$
z=6U_t\text{ (the newly sampled point)},\quad a=\min y<1 \text{ (the leftmost point)}
$$
Note also that
\begin{align}\label{pcond}p\ge \frac{N}{2(N-1)}=\frac{M+1}{2M}=\frac12 +\frac 1{2M}.
\end{align}
Define
\begin{equation*}
\begin{array}{rlrl}
F_{old}&=F(y), &
F_{new}&=F\left((y\cup\{z\})' \right)\\
\mu'_{old} &=\mu(y),& 
\mu'_{new} &=\mu\left((y\cup\{z\})' \right),\\
h_{old}&=F_{old}+k\left(\mu_{old}'\right)^2 , &
h_{new}&=F_{new}+k \left(\mu_{new}'\right)^2
\end{array}
\end{equation*}  
Thus we need to establish 
\begin{align}\label{eqsnos}
\E[h_{new}-h_{old}\|\F_t]\le 0.
\end{align}

First of all, observe that if $\tilde y=(y\setminus \{y_i\})\cup \{z\}$, that is, $\tilde y$ is obtained from $y$ by replacing $y_i$ with $y_0$, then
\begin{align*}
F(\tilde y)-F(y)&=\frac{z-y_i}{M}  
\left[(M-1)z+(M+1)y_i -2M \mu(y)\right] \\
\mu(\tilde y)^2-\mu(y)^2&=\frac{z-y_i}{M^2}  
\left[z-y_i +2M \mu(y)\right] 
\end{align*}
In particular, if we replace point $a$ by the new point $z$, then
$$
\Delta_a(z):=h_{new}-h_{old}=
\frac{z-a}{M}  
\left[(M-1)z+(M+1)a -2M \mu(y)  
+\frac{k}M   (z-a +2M \mu(y))\right] 
$$
and if we replace point $1$, then
$$
\Delta_1(z):=h_{new}-h_{old}=
\frac{z-1}{M}  
\left[(M-1)z+(M+1) -2M \mu(y)  
+\frac{k}M (z-1 +2M \mu(y))\right] 
$$
Note that both $\Delta_a$ and $\Delta_1$ depend only on four variables $(a,z,\mu,M)$ but not the whole configuration.
Let us also define 
$$ m(z)=p\cdot \frac{y_1+\dots+y_M+z}{M+1}=p\cdot \frac{M\mu+z}{M+1},
$$
the $p-$centre of mass of the old core and the newly sampled point.

There are three different cases that can occur: either (a) the point $a$ is removed, (b) $1$, the rightmost point of the previous core, is removed, or (c) the newly sampled point $z$ is removed. In the third case the core remains unchanged, and the change in the value of the function $h$ is trivially zero.
The point $a$ can only be removed if $z>a$; the point $1$ can only be removed if $z<1$; the point $z$ can be possibly removed only if $z\in(0,a)$ or $z\in(1,\infty)$.  Let us compute the critical values for~$z$, for which there is a tie between the farthest points.

\subsubsection*{Which point to remove?}
\noindent $(i)$ 
Suppose $\boxed{z<a}$.
Then there is a tie between~$z$ and~$1$ if and only if $m(z)=\frac{z+1}2$, that is if 
$$
z=t_{z1}:=\frac{M(2p\mu -1)-1}{M+1-2p}
\in
\begin{cases}
(-\infty,0) &\text{ if }p<p_1:=\frac{M+1}{2M\mu}\\
(0,a) &\text{ if }p_1<p<p_2:=\frac{(M+1)(a+1)}{2M\mu+2a}\\
(a,+\infty) &\text{ if }p>p_2.
\end{cases}
$$
Thus,  we have:
\begin{itemize}
\item when $p<p_1$,  point $1$ is removed;
\item when $p_1<p<p_2$, if $z<t_{z1}$ then $z$ is removed; if $z>t_{z1}$ point $1$ is removed;
\item when $p>p_2$, point $z$ is removed.
\end{itemize}

\noindent $(ii)$ 
Suppose $\boxed{a<z<1}$.
There is a tie between~$a$ and~$1$ if and only if $m(z)=\frac{a+1}2$, that is if 
$$
z=t_{a1}:=\frac{(M+1)(a+1)-2M\mu p}{2p}
\in
\begin{cases}
(1,+\infty) &\text{ if }p<p_3:=\frac{(M+1)(a+1)}{2M\mu+2},\\
(a,1) &\text{ if }p_3<p<p_2,\\
(-\infty,a) &\text{ if }p>p_2.
\end{cases}
$$
Thus,  we have:
\begin{itemize}
\item when $p<p_3$,  point $1$ is removed;
\item when $p_3<p<p_2$, if $z<t_{a1}$ then $1$ is removed; if $z>t_{a1}$ then point $a$ is removed;
\item when $p>p_2$, point $a$ is removed.
\end{itemize}

\noindent $(iii)$ 
Suppose $\boxed{z>1}$.
There is a tie between~$z$ and~$a$  if and only if $m(z)=\frac{z+a}2$, that is if 
$$
z=t_{za}:=\frac{2M\mu p-(M+1)a}{M+1-2p}
\in
\begin{cases}
(-\infty,1) &\text{ if }p<p_3,\\
(1,+\infty) &\text{ if }p>p_3.
\end{cases}
$$
Thus,  we have:
\begin{itemize}
\item when $p<p_3$,  point $z$ is removed;
\item when $p>p_3$, if $z<t_{za}$ then $a$ is removed; if $z>t_{za}$ then  point $z$ is removed.
\end{itemize}

We always have $p_1<p_2$, $p_3<p_2$ since
\begin{align*}
p_2-p_1&=\frac{a(M+1)(M\mu-1)}{2M\mu(M\mu+a)}=\frac{a(M+1)(a+(M-2)f)}{2M\mu(M\mu+a)}>0,\\
p_2-p_3&=\frac{(1-a)^2 (M+1)}{2(M\mu+1)(M\mu+a)}>0,
\end{align*}
while 
$$
p_1<p_3\  \Longleftrightarrow\  Ma\mu>1
 \Longleftrightarrow\  
 f>
\frac{1-a- a^2(M-1) }{a(M-2)(1-a) }
 \text{ (when $M>2$)}
$$

The final observation is that $t_{za}<6$, so there is indeed no need to sample the new point outside of the range $(0,6)$; this holds since $M\ge 2$ and
\begin{align*}
6-t_{za}&=\frac{-2 p (M\mu+6) +M a+6 M+a+6}{M+1-2p}
>\frac{-2 M\mu +M a+6 M+a-6}{M+1-2p}
\\
&
>\frac{-2 M\mu +6 M-6}{M+1-2p}
=\frac{2M(1- \mu)+4M-6}{M+1-2p}>
\frac2{M+1-2p}>0.
\end{align*}

\subsection*{The five cases for the removal:}
\begin{itemize}
 \item $p<\min\{p_1,p_3\}$:
 \begin{itemize}
  \item when $z<1$, point $1$ is removed
  \item when $z>1$, point $z$ is removed
 \end{itemize}

 \item $p>p_2$:
 \begin{itemize}
  \item when $z<a$ or $z>t_{za}\in(1,\infty)$ point $z$ is removed
  \item when $a<z<t_{za}$, point $a$ is removed
 \end{itemize}

 \item $\max\{p_1,p_3\}<p<p_2$
 \begin{itemize}
  \item when $z<t_{z1}\in(0,a)$ or $t>t_{za}\in(1,+\infty)$, point $z$ is removed 
  \item when $t_{z1}<z<t_{a1}\in(a,1)$, point $1$ is removed
  \item when $t_{a1}<z<t_{za}$, point $a$ is removed
 \end{itemize}

 \item $p_1<p<p_3\ (<p_2)$:
 \begin{itemize}
  \item when $z<t_{z1}\in(0,a)$ or $z>1$, point $z$ is removed
  \item when $t_{z1}<z<1$, point $1$ is removed
 \end{itemize}

 \item $p_3<p<p_1\ (<p_2)$:
 \begin{itemize}
  \item when $z<t_{a1}\in(a,1)$, point $1$ is removed 
  \item when $t_{a1}<z<t_{za}\in(1,+\infty)$, point $a$ is removed
  \item when $z>t_{za}$, point $z$ is removed
 \end{itemize}

\end{itemize}

Let 
\begin{align*}
X_1&=p-p_1=\frac{M(2\mu p-1)-1}{2M \mu},\\
X_2&=p-p_2=\frac{2ap-a-1+(2\mu p-a-1)M}{2(M\mu+a)},\\
X_3&=p-p_3=\frac{2p-a-1+ (2\mu p-a-1)M }{2(M\mu+1)}.
\end{align*}
Define
\begin{align*}
\tilde{\bf I}_1&=\left.  \E(h(\XX'(t+1))-h(\XX'(t))\|\F_t) \right|_{ x(t)=y} \cdot 1_{X_1<0} \cdot 1_{X_3<0} ,
\\  \tilde{\bf I}_2&=\left. \E(h(\XX'(t+1))-h(\XX'(t))\|\F_t) \right|_{ x(t)=y}  \cdot 1_{X_2>0} ,
\\  \tilde{\bf I}_3&\left.= \E(h(\XX'(t+1))-h(\XX'(t))\| \F_t) \right|_{ x(t)=y}  \cdot 1_{X_2<0} \cdot 1_{X_1>0} \cdot 1_{X_3>0},
\\ \tilde{\bf I}_4&= \left. \E(h(\XX'(t+1))-h(\XX'(t))\|\F_t) \right|_{ x(t)=y}  \cdot 1_{X_1>0} \cdot 1_{X_3<0},
\\  \tilde{\bf I}_5&= \left.\E(h(\XX'(t+1))-h(\XX'(t))\| \F_t) \right|_{ x(t)=y}  \cdot 1_{X_1<0} \cdot 1_{X_3>0}.
\end{align*}
Since $\max y=1$, because of the comment on the restriction of the uniform distribution on a subinterval, we have $\tilde{\bf I}_j=c_j {\bf I}_j$, $j=1,2,3,4,5$, where $c_j$'s are some positive constants and
\begin{align*}
\begin{array}{rclrcl}
{\bf I}_1&=& {\bf A}_1 \cdot 1_{X_1<0} \cdot 1_{X_3<0}, 
&
{\bf A}_1&=& \displaystyle\int_0^1 \Delta_1 dz  ,
\\
{\bf I}_2&=& {\bf A}_2 \cdot 1_{X_2>0}, 
& 
{\bf A}_2&=&\displaystyle\int_a^{t_{za}} \Delta_a dz ,
\\
{\bf I}_3&=& {\bf A}_3 \cdot 1_{X_2<0} \cdot 1_{X_1>0} \cdot 1_{X_3>0},
&
{\bf A}_3&=&\displaystyle\int_{t_{z1}}^{t_{a1}} \Delta_1 dz + \int_{t_{a1}}^{t_{za}} \Delta_a dz,  
\\
{\bf I}_4&=&{\bf A}_4  \cdot 1_{X_1>0} \cdot 1_{X_3<0},
&
{\bf A}_4&=& \displaystyle\int_{t_{z1}}^{1} \Delta_1 dz,
\\
{\bf I}_5&=&  {\bf A}_5\cdot 1_{X_1<0} \cdot 1_{X_3>0},
&
{\bf A}_5&=&\displaystyle\int_{0}^{t_{a1}} \Delta_1 dz + \int_{t_{a1}}^{t_{za}} \Delta_a dz. 
\end{array}
\end{align*}
Thus to establish~\eqref{eqsnos}, it suffices to show that ${\bf I}_j\le 0$ for each $j=1,2,3,4,5$. This is done by very extensive and tedious calculations, which can be found in the Appendix.
\end{proof}

We now return to our original $p$-contest process $\XX(t)$. For $L\ge 2$ define
\begin{align*}
\tau_L&=\inf\{t>0:\ x_{(N-1)}(t)<1/L\};\\
\eta_L&=\inf\{t>\tau_L:\ x_{(N-1)}(t)\ge 1/2\},
\end{align*} 
note that $\tau_L$  is a.s.\ finite for every $L$ by Proposition~\ref{liminfzero}. Let $W(s)=\{w_1(s),\dots,w_N(s)\}$ be a borderless $p$-contest with $W(0)=\XX(\tau_L)$; let $W'(s)$ be its core.
By Lemma~\ref{lemMod} the quantity $\xi_t=h(W'(t\wedge \eta_L))$ is a supermartingale,  that converges to some $\xi_\infty$.
Since $\xi_t$ is bounded,
\begin{align*}
\E \xi_0\ge \E \xi_\infty=\E [\xi_\infty\cdot  1_{\eta_L<\infty}]
+\E [\xi_\infty\cdot 1_{\eta_L=\infty}]
\ge \E [\xi_\infty\cdot 1_{\eta_L<\infty}]
\ge \frac{k}{(2(N-1))^2}\P(\eta_L<\infty)
\end{align*}
since on $\{\eta_L<\infty\}$ we have $\xi_\infty=W'( \eta_L)$  and the largest coordinate of $W'( \eta_L)$ is larger than $1/2$, implying that $\mu(W'(\eta_L))\ge \frac 1{2(N-1)}$ and thus $h(W'(\eta_L))=F(W'(\eta_L))+k\mu(W'(\eta_L))^2 \ge \frac k{(2(N-1))^2}$.
We also have
$$
\xi_0=h(\XX'(\tau_L))=F(\XX'(\tau_L))+k\mu(\XX'(\tau_L))^2 \le \frac{N-1}{L^2} +\frac{k}{L^2}
\quad
\Longrightarrow
\quad
\E \xi_0\le \frac{N+k-1}{L^2}
$$
since $\XX'(\tau_L)\subset [0,1/L]$ and so $\mu(\XX'(\tau_L))\in [0,1/L]$.

Combining the above inequalities, we conclude that $\P(\eta_L<\infty)\to 0$ as $L\to\infty$.
However, on $\eta_L=\infty$ the core of the regular $p$-contest process can be trivially coupled with the core of the borderless process $W'(s)$ which converges to zero, so  $\XX'(t)\to 0$ as well. Since $\P(\eta_L=\infty)$ can be made arbitrarily close to $1$ by choosing a large $L$, we conclude that $\XX'(t)\to 0$ a.s.
\end{proof}

\section{The case $p>1$}
Throughout this section  we suppose that $\zeta$ has a full support on $[0,1]$, and, unless explicitly stated otherwise, that $p>1$.

\begin{thm}\label{thmpgreater}
\begin{itemize}
\item[(a)]
$
 \P\left(\{ \XX'(t)\to 0\}\cup \{ \XX'(t)\to 1\} \right)=1;
$
\item[(b)]
if $x_{(1)}(0)\ge 1/p$ then  $\P( \XX'(t)\to 1)=1$;
\item[(c)]
if $x_{(k)}(0)>0$, where $k$ satisfies
\begin{align}\label{eqksmall}
\left\{2p(N-k)>N-2p \right\}\Longleftrightarrow \left\{k<N-\frac{N}{2p}+1\right\},
\end{align}
then $ \P\left( \XX'(t)\to 1 \right)>0.$
\end{itemize}

\end{thm}

\begin{rema}
In general, both convergences can have a positive probability. Let $N=3$, $p\in(1,3/2)$, and
$$
\zeta=\begin{cases}
U,&\text{ with probability }1/3;\\
0,&\text{ with probability }1/3;\\
1,&\text{ with probability }1/3,
\end{cases}
$$
where $U\in U[0,1]$ (so $\zeta$ has full support). Suppose we sample the points of $\XX(0)$ from~$\zeta$.
If they all start off in $0$,  then $p\mu \le  p/3 < 1/2$, so they cannot escape from $0$. On the other hand, there is a positive probability they all start in $(1/p, 1]$, and then Theorem 2(b) says that they converge to $1$.
\end{rema}

The key idea behind the proof of Theorem~\ref{thmpgreater} is that one can actually find the ``ruling" order statistic of the core; namely, there exists some non-random $k=k(N,p)\in\{1,2,\dots, N-1\}$ such that $x_{(k)}(t)\to 0$  implies $\XX'(t)\as 0$, while $x_{(k)}(t)\not\to 0 $  implies that $\XX'(t)\as 1$.

We start with the following two results, which tells us that there is {\em an absorbing area}  $[\frac{1}{p},1]$ for the process, such that, once the core enters this area, it will never leave it, and moreover the core will keep moving to the right.
\begin{claim}\label{absorbing}
Suppose that  $x_1\le x_2\le x_3\le \dots\le x_{N}\le 1$  and $x_2\ge p^{-1}$. Then  $\{x_1,\cdots,x_N\}'=\{x_2,\cdots,x_N\}$
\end{claim}
\begin{proof}
Let $\mu=\frac{x_1+\cdots x_N}{N}$. If $p\mu \ge x_N$ then the claim follows immediately; assume  instead that $p\mu < x_N$. We need to check if $p\mu-x_1>x_N-p\mu$, that is, if
\begin{align}\label{eqx2p}
2p(x_2+\cdots+x_{N-1})>(N-2p)(x_1+x_N)
\end{align}
However, since $x_i\ge x_2$ for $i=3,\dots, N-1$ we have
$$
2p(x_2+\cdots+x_{N-1})\ge 2p x_2 (N-2)\ge 2(N-2)
$$
while $(N-2p)(x_1+x_N)\le 2(N-2p)<2(N-2).$
Hence~\eqref{eqx2p} follows.
\end{proof}

\begin{lemma}\label{lem_all_above_p}
If $x_{(1)}(t_0)\ge 1/p$ for some $t_0$, then $\XX'(t)\to 1$ a.s.
\end{lemma}
\begin{proof}
If $x_{(1)}(t_0) \ge 1/p$, then any point that lands in $[0, 1/p)$ is extreme, so $x_{(2)}(t)\ge 1/p$ for all $t\ge t_0$.
Choose any positive $\eps<1-\frac 1p$, and let $A_t=\left\{\zeta_{t+1},\dots,\zeta_{t+N-1}\in\left(1-\eps,1\right]\right\}$.
Then if~$A_t$ happens for  $s > t_0$, any point in $\left[0, 1 - \eps\right]$ is removed in preference to any of the new points coming in, so $x_{(2)}(s+N-1) > 1 -\eps$. As a result, by Claim~\ref{absorbing} we get that $\XX'(t)\in \left[0, 1 - \eps\right]$ for all $t\ge s$.

On the other hand,  $\P(A_t)\ge [f(\eps)]^{N-1}>0$ (see~\eqref{eqfullsupp}) for any $t$, and the events $A_t,A_{t+N},A_{t+2N},\dots$ are independent. Hence, eventually with probability $1$,  one of the $A_t$'s must happen for some $t>t_0$, so a.s.\ $\XX'(t)\in \left[0, 1-\eps\right]$ for all large $t$. Since $\eps$ can be chosen arbitrary small, we get the result.
\end{proof}

The next two results show that if the is some $\eps>0$ such that infinitely often the core does not have any points in $[0,\eps)$, then it must, in fact, converge to $1$.
\begin{lemma}\label{lem_send_to_above_p}
If $x_{(1)}(t_0)\ge \eps$ for some $t_0$ and $\eps>0$, then $\P(x_{(1)}(t_0+\ell)\ge p^{-1}\|\F_t)\ge \delta$ for some $\ell=\ell(\eps)$ and $\d=\d(\eps)>0$.
\end{lemma}
\begin{proof}
Suppose that for some $t$ we have $x_{(1)}(t)\ge \eps$. We claim that  it is possible to move $x_{(1)}$ to the right of $\frac{1+p}2\eps$ in at most $N-1$ steps  with positive probability, depending only on $p$ and $\eps$. Indeed, if $x_{(1)}(t)>\frac{1+p}2\eps$ then we are already done. 
Otherwise, if the new point $\zeta_{t+1}$ is sampled in $\left(\frac{1+p}2 \eps,p\eps \right]\subset [0,1]$ it cannot be rejected. If at this stage $x_{(1)}(t+1)>\frac{1+p}2\eps$, then we are done. If not, we proceed again by sampling $\zeta_{t+2} \in \left(\frac{1+p}2\eps,p\eps\right]$, etc. After at most $N-1$ steps of sampling new points in $\left(\frac{1+p}2\eps,p\eps\right]$, the leftmost point $x_{(1)}$ will have moved to the right of $\frac{1+p}2\eps$. 

Thus, in no more than $N-1$ steps, with probability no less than $\left[f\left(\frac{p-1}2\eps\right)\right]^{N-1}>0$, $x_{(1)}$ is to the right of $\frac{1+p}2\eps$.
By iterating this argument at most $m$ times, where $m\in\N$ is chosen such that $\left[\frac{1+p}{2}\right]^m \, \eps>1/p$, we achieve that $x_{(1)}$ is to the right of $1/p$ (for definiteness, one can chose $\ell=(N-1)m$ and $\d=
\left[f\left(\frac{p-1}2\eps\right)\right]^{(N-1)m}$.)
\end{proof}

\begin{lemma}\label{absorbinglemma}
Let $\eps\in(0,1)$,
and define $B(\eps):=
\{x_{(1)}(t)\ge \eps \text{ i.o.}\}$
Then
$
B(\eps) \imp \{ \XX'(t)\to 1\}.
$
\end{lemma}

\begin{corollary}\label{absorbingcor}
We have $\left\{\liminf_{t\to\infty} x_{(1)}(t)>0\right\}\imp \{ \XX'(t)\to 1\}.$
\end{corollary}

\begin{proof}[Proof of Lemma~\ref{absorbinglemma}]
Assume that $\eps<\frac 1p$ (otherwise the result immediately follows from Lemma~\ref{lem_all_above_p}). Also suppose
that $\P\left(B(\eps)\right)>0$, since otherwise the result is trivial. Let $\ell$ and $\d$ be the quantities from Lemma~\ref{lem_send_to_above_p}.

Define
\begin{align*}
\tau_0&=\inf\{t>0:\ x_{(1)}(t)>\eps\},\\
\tau_k&=\inf\{t>\tau_{k-1}+\ell:\ x_{(1)}(t)>\eps\},\quad k\ge 1,
\end{align*}
with the convention that if $\tau_k=\infty$ then $\tau_m=\infty$ for all $m>k$. Notice that $B(\eps)=\bigcap_{k=0}^\infty\{\tau_k<\infty\}$. 
On $B(\eps)$ we can also define $D_{\tau_k}=\{x_{(1)}(\tau_k+\ell)\ge 1/p\}$.
Since $\tau_k-\tau_{k-1}>\ell$ whenever both are finite, we have from Lemma~\ref{lem_send_to_above_p} we have
$\P(D_{\tau_{k+1}} \| \F_{\tau_k}) \ge \d $.
Therefore,
$$
B(\eps)\imp \left\{
\sum_{k\ge 0} 
\P(D_{\tau_{k+1}} \| \F_{\tau_k})=\infty
\right\}
$$
hence by  L\'evy's extension of the Borel-Cantelli lemma it follows that a.s.\
on $B(\eps))$ infinitely many (and hence at least one) of $D_{\tau_k}$ occur, that is, $x_{(1)}(\tau_k+\ell)\ge 1/p$. Now the result follows from Lemma~\ref{lem_all_above_p}.
\end{proof}

Assume for now that $p<\frac{N}2$; in this case $N-\frac{N}{2p}+1<N$ (see~\eqref{eqksmall}). The case $p\ge \frac{N}2$ will be dealt with separately.

The following statement shows that if all the points to the right of $x_{(k)}$ lie very near each other, while the left-most one lies near zero, then it is to be removed.
\begin{claim}\label{dragto1Claim}
Let $a\in(0,1]$ and suppose that  $k\in\{2,\dots,N-1\}$ satisfies~\eqref{eqksmall}. Then there exist small $\d,\Delta>0$, depending on $N,k,p,a$ such that if  
\begin{align*}
0&\le x_1\le \d;\\
x_1&\le x_i\le  x_N\quad \text{for }i=2,\dots,N-1;\\
x_k,x_{k+1},\dots,x_N&\in [a(1-\Delta),a)
\end{align*}
then $\{x_1,\dots,x_{N}\}'=\{x_2,\dots,x_N\}$.
\end{claim}
\begin{proof}
The condition to remove the leftmost point is $p\mu-\frac{x_1+x_N}2> 0$ where $\mu=(x_1+\dots+x_N)/N$.
However,
\begin{align*}
2N \left(p\mu-\frac{x_1+x_N}2\right)&=2p(x_2+\dots +x_{N-1})-(N-2p)x_1-(N-2p) x_N \\
& \ge 2p(x_k+\dots +x_{N-1})-(N-2p)\delta -(N-2p) a
\\
& \ge 
2p(N-k)a(1-\Delta)-(N-2p)\delta -(N-2p) a
\\
&=a\left[2p(N-k)(1-\Delta)-(N-2p) \right]-(N-2p)\delta
\end{align*}
The RHS is linear in $\d$ and $\Delta$, and when $\d=\Delta=0$ it is strictly positive by the assumption on~$k$; hence it can also be made positive, by allowing $\d>0$ and $\Delta>0$ to be sufficiently small.
\end{proof}
\begin{corollary}\label{cordelta}
Suppose that $\XX(t)=\{x_1,\dots,x_N\}$ satisfies the conditions of Claim~\ref{dragto1Claim} for some $a$ and $k$. Let $\delta$ be the quantity from this claim. Then 
$$
\P(x_{(1)}(t+j)>\delta\text{ for some }1\le j\le k\|\F_t)\ge c=c_{a\Delta}>0.
$$ 
\end{corollary}
\begin{proof}
The probability to sample a new point $\zeta\in (a(1-\Delta),a]$ is bounded below by $f(a\Delta)$ where~$f$ is the same function as in~\eqref{eqfullsupp}. On the other hand, if the new point is sampled in $(a(1-\Delta),a]$ then $\XX(t+1)$ continues to satisfy the conditions of Claim~\ref{dragto1Claim} as long as the leftmost point is in $[0,\delta]$. By repeating this argument at most $k$ times and using  the induction, we get the result with $c=\left[f(a\Delta)\right]^k>0$.
\end{proof}

\begin{lemma}\label{dragto1Lemma}
Let $k\in\N$ satisfy~\eqref{eqksmall}. Then
$$
\left\{ x_{(k)}(t)\not \to 0 \right\}\imp \left\{ \XX'(t) \to 1 \right\}.
$$
\end{lemma}
\begin{proof}
Note that by Lemma~\ref{absorbinglemma}, it suffices to show that  $\left\{ x_{(k)}(t)\not \to 0 \right\}\imp \left\{x_{(1)}(t)\not \to 0\right\}$.

If $x_{(k)}(t)\not\to 0$, there exists an $a>0$ such that $x_{(k)}(t)\ge a$ for infinitely many $t$'s. Let $s$ be such a time. Now suppose that $\zeta_{s+i}\in I:=(a(1-\Delta),a]$ for $i=0,1,\dots, N-1$ where $\Delta$ is defined in Claim~\ref{dragto1Claim}; the probability of this event is strictly positive and depends only on~$a$ and~$\d$ (see~\eqref{eqfullsupp}). As long as there are points of $\XX(s+i)$ on {\em both} sides of the interval $I$, none of the points inside $I$ can be removed; hence, for some $u\in\{s,s+1,\dots,s+N-1\}$ we have that
either $\min \XX(u)> a(1-\Delta)$ or $\max \XX(u)\le  a$. In the first case, $x_{(1)}(u)>a(1-\Delta)$. 

In the latter case, both $x_{(N)}(u)\in I$ and $x_{(k)}(u)\in I$, since every time we replaced a point, the number of points to the left of $I$ did not increase (and there were initially at most $k-1$ of them). As a result
$$
a(1-\Delta)\le x_{(k)}(u)\le x_{(k+1)}(u)
\le \dots \le x_{(N)}(u)\le a.
$$
Together with Corollary~\ref{cordelta}, this yields
$$
\{x_{(k)}(t)\ge a\text{ i.o.}\}
\imp
\{x_{(1)}(t)\ge \min\{a(1-\Delta),\delta\}\text{ i.o.}
\}\imp \{x_{(1)}(t)\not\to 0\}
$$
which proves Lemma~\ref{dragto1Lemma}.
\end{proof}
\vskip 1cm

\begin{claim}\label{claimcriticalp}
Let $A_i:=\left\{x_{(i)}(t)\to 0\right\}$ and suppose that 
for some $1\le k\le N-2$ we have
\begin{align}\label{eqkspecial}
\left\{2p(N-k-1)<N \right\}\Longleftrightarrow \left\{ k>N-\frac{N}{2p}-1\right\}.
\end{align}
Then 
$
A_k\subseteq \{\exists \lim_{t\to\infty} x_{(k+1)}(t)\}.
$
\end{claim}
\begin{proof}
Fix any $a>0$.
Let $\d>0$ be so small that
\begin{align}\label{eq2pN}
2pN\d <[N-2p(N-k-1)]a.
\end{align}
In the event $A_k$ there exists a finite $\tau=\tau_\d(\omega)$ such that 
$$
\left\{\sup_{t\ge \tau} x_{(k)}(t)\le \d\right\}
\Longleftrightarrow
\left\{ {\rm card}\,\left( \XX'(t) \cap [0,\d]\right)\ge k\text{ for all }t\ge \tau.\right\}
$$
From now on assume that $t\ge \tau$. We will show below that $x_{(k+1)}(t+1)\le \max\{x_{(k+1)}(t),a\}$.

To begin, let us prove that $x_{(k+1)}(t+1)\le x_{(k+1)}(t)$ as long as $x_{(k+1)}(t)>\d$.
Indeed, if the new point $\zeta$ is sampled to the left of $x_{(k+1)}(t)$, then regardless of which point is to be removed, $x_{(k+1)}(t+1)\le x_{(k+1)}(t)$. If the new point $\zeta$ is sampled to the right, then the farthest point from the $p-$centre of mass must be the rightmost one (and hence $x_{(k+1)}(t+1)=x_{(k+1)}(t)$) since there are exactly $k$ points in $[0,\d]$ and none of these can be removed by the definition of $\tau$.

On the other hand, if $x_{(k+1)}(t)\le \d$ then either $x_{(k+2)}(t)\le a$ or $x_{(k+2)}(t)> a$. In the first case, $x_{(k+1)}(t+1)\le x_{(k+2)}(t)\le a$ even if $x_{(1)}$ is removed. In the other case, when $x_{(k+2)}(t)> a$, we have $x_{(N-1)}>a$ as well, and
\begin{align*}
p\mu(\XX(t+1))&\le p\frac{(k+1)\d+(N-k-1)x_{(N)}}{N}<
\frac{2pN\d-[N-2p(N-k-1)]x_{(N)} +N x_{N}}{2N}
\\
&\le 
\frac{N x_{N} -\left\{[N-2p(N-k-1)]a-2pN\d\right\})}{2N}
<\frac{x_{(N)}}2
\end{align*}
by~\eqref{eq2pN}, so $x_{(N)}=x_{(N)}(t)$ must be removed and thus $x_{(k+1)}(t+1)\le x_{(k+1)}(t)$.

Consequently, we obtained
\begin{align*}
A_k&\subseteq \bigcap_{t\ge\tau} \left\{x_{(k+1)}(t+1)\le \max\{x_{(k+1)}(t),a \}\right\}
\\
&\subseteq
\left(\bigcup_{t\ge 0} 
\left\{x_{(k+1)}(s)\le a \text{ for all }s\ge t\right\}
\right)
\cup
\left(\bigcup_{t\ge 0} \left\{x_{(k+1)}(s)\le x_{(k+1)}(s+1) \text{ for all }s\ge t\right\}\right)
\\
&\subseteq
\left\{\limsup_{t\to\infty} x_{(k+1)}(t)\le a\right\}
\cup
\left\{\exists \lim_{t\to\infty} x_{(k+1)}(t)\right\}
\end{align*}
since $a>0$ is arbitrary, we get 
$$
A_k\subseteq 
\left\{\limsup_{t\to\infty} x_{(k+1)}(t)\le 0\right\}
\cup
\left\{\exists \lim_{t\to\infty} x_{(k+1)}(t)\right\}
=\left\{\exists \lim_{t\to\infty} x_{(k+1)}(t)\ge 0\right\}
$$
\end{proof}

\begin{lemma}\label{lemintN2p}
Suppose that~\eqref{eqkspecial} holds for some $1\le k\le N-2$. Then $A_k \imp A_{k+1}$.
\end{lemma}
\begin{proof}
Let ${\tilde A}_{k+1}^{\ge a}:=\left\{\lim_{t\to\infty} x_{(k+1)}(t)\ge a\right\}$ (the existence of this limit on $A_k$ follows from Claim~\ref{claimcriticalp}).
It suffices to show that $\P\left(A_k \cap \tilde A_{k+1}^{\ge a}\right)=0$ for all $a>0$; then from the continuity of probability we get that $\P\left(A_k \cap \{\lim_{t\to\infty} x_{(k+1)}(t)>0\}\right)=0$ and hence $A_k\imp A_{k+1}$.

Fix an $a>0$. Let
\begin{align*}
C_t&=\left\{x_{(k)}(t)< \frac a3\text{ and }x_{(k+1)}(t)> \frac{2a}3\right\},\qquad 
\bar C_T=\bigcap_{t\ge T} C_t,
\end{align*}
then
$$
A_k \cap \tilde A_{k+1}^{\ge a}\subseteq \bigcup_{T\ge 0} \bar C_T
=
\left\{\exists T>0:\ x_{(k)}(t)< \frac a3\text{ and }x_{(k+1)}(t)> \frac{2a}3\text{ for all }t\ge T\right\}.
$$
If the probability of the LHS is positive, then, using the continuity of probability and the fact that $\bar C_T$ is an increasing sequence of events, we obtain that $\lim_{T\to\infty} \P(\bar C_T)>0$. Consequently, there exists a {\em non-random} $T_0$ such that $\P(\bar C_{T_0})>0$.

This is, however, impossible, as at each time point $t$, with probability at least $f(a/3)$ (see~\eqref{eqfullsupp}) the new point $\zeta_t$ is sampled in $B:=\left(\frac{a}3,\frac{2a}3\right)$ and then either $x_{(k)}(t+1)\in B$ or $x_{(k+1)}(t+1)\in B$.
Formally, this means that
$$
\P(C_{t+1} \| C_t,\F_t)\le 1-f(a/3)\quad\text{for all }t\ge 0.
$$
By induction, for all $k\ge 1$,
$$
\P(\bar C_{T_0}\|\F_{T_0})\le \P\left(\bigcap_{T=T_0}^{T_0+k} C_t \| \F_{T_0}\right)
\le \left[1-f(a/3)\right]^k.
$$
Since $k$ is arbitrary, and $f(a/3)>0$, by taking the expectation, we conclude that $\P(\bar C_{T_0})=0$ yielding a contradiction.

Hence the probability of the event $A_k\cap \tilde A_{k+1}^{\ge a}$ is zero.
\end{proof}

\begin{corollary}\label{corintN2p}
Suppose that~\eqref{eqkspecial} holds for some $1\le k\le N-2$. Then 
$$
\left\{ x_{(k)}(t) \to 0 \right\}\imp \left\{ \XX'(t) \to 0 \right\}.
$$
\end{corollary}
\begin{proof}
Observe that if $k$ satisfies~\eqref{eqkspecial} then $k+1$ satisfies~\eqref{eqkspecial} as well. Thus by iterating Lemma~\ref{lemintN2p} we obtain that $A_k\imp A_{k+1}\imp A_{k+2}\imp\dots\imp A_{N-1}$, i.e. $x_{(N-1)}(t)\to 0$, which is equivalent to the statement of Corollary.
\end{proof}
\begin{rema}\label{extraplessone}
Note that the condition~\eqref{eqkspecial} does not assume $p>1$; hence the conclusion of Corollary~\ref{corintN2p} holds  for the case $0<p\le 1$ as well.
\end{rema}

\vskip 1cm
For the case $p\ge \frac{N}2$ we have
\begin{lemma}\label{lempN2}
If $p\ge \frac{N}2$ then $\XX'(t)\to 1$ a.s.
\end{lemma}
\begin{proof}
The case $p>\frac{N}2$ is easy: with a positive probability the newly sampled point $\zeta>0$ and then
$$
p\ \frac{x_{(1)}+\dots +x_{(N-1)}+\zeta}{N}
>\frac{x_{(1)}+\dots+x_{(N-1)}+\zeta}2\ge \frac{x_{(1)}+\zeta}2
$$
hence it is the left-most point which is always removed, implying $\liminf_{t\to\infty} x_{(1)}(t)>0$. Hence by Corollary~\ref{absorbingcor}, $\XX'(t)\to 1$ a.s.

For the case $p=\frac{N}2$ we notice that at each moment of time we either have a tie (between the left-most and right-most point) or remove the left-most point. However, we can only have a tie if $x_{(1)}(t)=...=x_{(N-1)}(t)=0$; in this case,  eventually the right-most point will be kept and the left-most removed. After this moment of time, there will be more ties, and the left-most point will always be removed, leading to the same conclusion as in the case $p>N/2$.
\end{proof}

\begin{proof}[Proof of Theorem~\ref{thmpgreater}]
Part (b) follows from Lemma~\ref{lem_all_above_p}.

To prove part (c), note that unless $x_{(1)}(0)>0$ already, by repeating the arguments from the beginning of the proof of Lemma~\ref{dragto1Lemma}, with a positive probability we can ``drag" the whole configuration in at most $N-1$ steps to the right of zero, that is, there is $0\le t_0\le N-1$ such that $\P(\min\XX'(t_0)>0)>0$. Now we can apply Lemma~\ref{lem_send_to_above_p} and then Lemma~\ref{lem_all_above_p}.

Let us now prove part (a).
First, assume  $p<\frac{N}2$.
It is always possible to find an integer $k$ which satisfies both~\eqref{eqksmall} and~\eqref{eqkspecial}, so let $k$ be  such that
$$
N-\frac{N}{2p}-1<k<N-\frac{N}{2p}+1
$$
(if $N/(2p)\in\N$ this $k$ will be unique).
Now the statement of the theorem follows from Corollary~\ref{corintN2p} and Lemma~\ref{dragto1Lemma}.

Finally, in case $p\ge \frac{N}2$ the theorem follows from Lemma~\ref{lempN2}.
\end{proof}

\section{Non-convergence to zero for $p>1$ and $N=3$}
In this section we prove the following
\begin{thm}\label{thmnotzero}
Suppose that $N=3$, $p>1$ and  $\zeta$, restricted to  some neighbourhood of zero, is a continuous random variable with a non-decreasing density (e.g.\ uniformly distributed).
Then $\XX'(t)\to 1$ as $t\to\infty$ a.s.
\end{thm}
\begin{rema}\label{remN3}$ $
\begin{itemize}
\item In case $p\ge 3/2$ we already know that $\XX'(t)\to 1$ for any initial configuration and any distribution (see Lemma~\ref{lempN2}), so we have to prove the theorem only for $p\in(1,3/2)$.
\item
Simulations suggest that the statement of Theorem~\ref{thmnotzero} holds, in fact, for a much more general class of distributions $\zeta$.
\end{itemize}
\end{rema}

Let $\eps\in(0,1/2)$ be such that $\zeta$ conditioned\footnote{note that the full support assumption ensures that the probability of this event is positive} on $\{\zeta\le 2\eps\}$ has a non-decreasing density; according to the statement of the Theorem~\ref{thmnotzero} such an $\eps$ must exist. Let us fix this $\eps$ from now on.

The idea of the proof will be based on finding a non-negative function $h:(0,1]^2\to\R_+$ which has the following three properties: 
\begin{itemize}
\item[(i)] $h(\cdot,\cdot)$ is non-increasing in each of its arguments; \item[(ii)] $h\left(x_{(1)}(t),x_{(2)}(t)\right)$ is a supermartingale as long as $x_{(2)}(t)\le \eps$; 
\item[(iii)] $h(\cdot,\cdot)$ goes to infinity when the first coordinate goes to zero. 
\end{itemize}
From the supermartingale convergence theorem it would then follow that
$$
\P\left(\liminf_{t\to\infty} x_{(1)}(t)>0
\text{ or }
\limsup_{t\to\infty} x_{(2)}(t)\ge \eps \right)=1.
$$

Let us formally prepare for the proof of Theorem~\ref{thmnotzero}.
As before, denote by $x_1,\dots,x_N$  $N$ distinct points on $[0,1]$, and let $x_{(1)},\dots,x_{(N)}$ be this unordered $N$-tuple sorted in the increasing order. Let 
$$
\{y_1,\dots,y_{N-1}\}=\{x_1,\dots,x_N\}'_p
$$
be the unordered $N$-tuple $\{x_1,\dots,x_N\}$ with the farthest point from $p$-centre of mass removed; w.l.o.g.\ assume that $y_i$ are already in the increasing order.
\begin{lemma}\label{lemmon}
The operation $\{\dots\}'_p$ is monotone in $p$, that is, if $\hat p\ge \tilde p$ and
\begin{align*}
\{\hat y_1,\dots,\hat y_{N-1}\}&=\{x_1,\dots,x_N\}'_{\hat p},\\
\{\tilde y_1,\dots,\tilde y_{N-1}\}&=\{x_1,\dots,x_N\}'_{\tilde p}
\end{align*}
then $\hat y_i\ge \tilde y_i$, $i=1,\dots,N-1$.
\end{lemma}
\begin{proof}
Assume w.l.o.g.\ $x_1\leq ...\leq x_N$, and let $\mu = \mu\left(\{x_1,\dots,x_N\}\right)$. Notice that, regardless of the value of $p$,  the only points which can possibly be removed are $x_1$ or $x_N$ (since they are the two extreme points). Therefore, it suffices to show that $\{x_1,\dots,x_N\}'_{\tilde p}=\{x_2,\dots,x_N\}$ implies $\{x_1,\dots,x_N\}'_{\hat p}=\{x_2,\dots,x_N\}$.   Note also that $|x_1-p\mu|=p\mu-x_1$ for all $p\ge 1$.

If $\tilde p\mu-x_1>|\tilde p\mu-x_N|$ and $\tilde p\mu-x_N>0$, that is, the $p-$centre of mass lies to the right of $x_N$, then $\hat p\mu>\tilde p\mu>x_N$ as well,  and hence $x_1$ is discarded.

On the other hand, if $\tilde p\mu-x_1>|\tilde p\mu-x_N|$ and $\tilde p\mu<x_N$ then either $\hat p\mu<x_N$, or $\hat p\mu\ge x_N$. In the first case,
$$
\hat p\mu-x_1>\tilde p\mu-x_1>|\tilde p\mu-x_N|=x_N-\tilde p\mu
>x_N-\hat p\mu=|x_N-\hat p\mu|
$$
so $x_1$ is discarded. In the second case, $p-$centre of mass lies to the right of $x_N$ and so $x_1$ is also discarded.
\end{proof}

\begin{lemma}\label{lemsupermp}
Let $h$ be a real-valued function on the sets of $N$ real numbers. Suppose that $h$ is non-increasing in each of its arguments, namely
$$
h\left(x_1,x_2,\dots,x_{i-1},x_i',x_{i+1},\dots,x_N\right)
\le h\left(x_1,x_2,\dots,x_{i-1},x_i,x_{i+1},\dots,x_N\right)
$$
whenever $x_i'\ge x_i$. Let~${\cal E}_t$ be some $\F_t$-measurable event, and suppose that 
\begin{align}\label{eqsuper}
\E\left(h(\XX'(t+1))\|\F_t\right)\le h(\XX'(t))\text{ on }{\cal E}_t
\end{align}
for $p=1$. Then~\eqref{eqsuper} holds for $p>1$ as well.
\end{lemma}
\begin{proof}
Let 
$$
G_p(\XX'(t),\zeta_{t+1})=\{x_{(1)}(t),x_{(2)}(t),\dots,x_{(N-1)}(t),\zeta_{t+1}\}'_p
$$ 
be the new core after the new point $\zeta_{t+1}$ is sampled and the farthest point from the $p-$centre of mass is removed; note that $\XX'(t+1)=G_p(\XX'(t),\zeta_{t+1})$.
Then on ${\cal E}_t$
\begin{align*}
\E(h(\XX'(t+1))\|\F_t) = \E (h(G_p(\XX'(t),\zeta_{t+1}))\|\F_t)
\le
\E (h(G_1(\XX'(t),\zeta_{t+1}))\|\F_t) \le h(\XX'(t))
\end{align*}
since the operation $\{\dots\}'_p$ is monotone in $p$ by Lemma~\ref{lemmon} and $h$ is decreasing in each argument.
\end{proof}

From now on assume $N=3$ and $p=1$. Denote $x_{(1)}(t)=a$, $x_{(2)}(t)=b$ and consider the events 
$$
\begin{array}{rclrcl}
L_b&=&\left\{\zeta_{t+1}\in \left((2a-b)^+,a \right)\right\},
&
R_a&=&\left\{\zeta_{t+1} \in \left(b,2b-a \right) \right\},
\\
B_b&=&\left\{\zeta_{t+1} \in \left(a,\frac{a+b}{2} \right) \right\}, &
B_a&=&\left\{\zeta_{t+1} \in \left(\frac{a+b}{2},b \right) \right\}
\end{array}
$$
(we assume that $b$ is smaller than $1/2$, yielding $2b-a<1$.)
If $x_{(2)}(t)\le \eps$ then $\XX'(t+1)\not=\XX'(t)$ implies that one of the events $L_b$, $B_b$, $B_a$ or $R_a$ occurs (i.e. all points sampled outside of $\left((2a-b)^+,2b-a \right)$ are rejected at time $t+1$). Let us study the core $\XX'(t+1)=\{\zeta_{t+1},a,b\}'$ on these events: on $L_b$ and $B_b$ we have $\XX'(t+1)=\{x,a\}$, while on $B_a$ and $R_a$ we have $\XX'(t+1)=\{x,b\}$.

We have, assuming $x_{(1)}(t)=a$ and $x_{(2)}(t)=b$,
$$
\E(h(\XX'(t+1))-h(\XX'(t))\| \F_t)=\E (h\left( \{\zeta,a,b\}' \right)-h(a,b)).
$$
When $0\le a\le b\le \eps$ we have $2b-a\le 2\eps$.
Define
$$
g(x)=h\left( \{x,a,b\}' \right)-h(a,b)
=\begin{cases}
h(x,a)-h(a,b),&\text{if } x\in ((2a-b)^+,a);\\
h(a,x)-h(a,b),&\text{if } x\in (a,(a+b)/2);\\
h(x,b)-h(a,b),&\text{if } x\in ((a+b)/2,b);\\
h(b,x)-h(a,b),&\text{if } x\in (b,2b-a)\\
0, &\text{otherwise},
\end{cases}
$$ 
which is positive in the first two cases, and negative in the next two. Let $\ff(x)$ be the density of~$\zeta$ conditioned on $\{\zeta\in [0,2\eps]\}$. By the monotonicity of $\ff$ and the positivity (negativity resp.) of $g$ on the first (second resp.) interval,
\begin{align*}
\Delta(a,b) &:=\E \left[g(\zeta) 1_{\zeta\in [0,2\eps]}\right]=\int_{(2a-b)^+}^{\frac{a+b}{2}}g(x) \ff(x) {\rm d}x
+\int_{\frac{a+b}2}^{2b-a} g(x) \ff(x) {\rm d}x
\\
&
\le \ff\left(\frac{a+b}2\right)\int_{(2a-b)^+}^{\frac{a+b}{2}}g(x) {\rm d}x
+\ff\left(\frac{a+b}2\right)\int_{\frac{a+b}2}^{2b-a} g(x)  {\rm d}x
= \ff\left(\frac{a+b}{2}\right)\cdot \Lambda,
\end{align*}
 where 
\begin{align*}
\Lambda=&\Lambda(a,b)=\int_{(2a-b)^+}^a (h(x,a)-h(a,b)) {\rm d}x
+\int_a^{\frac{a+b}2} (h(a,x)-h(a,b)) {\rm d}x \\
& +\int_{\frac{a+b}2}^b (h(x,b)-h(a,b)) {\rm d}x
+\int_{b}^{2b-a}(h(b,x)-h(a,b)){\rm d}x.
\end{align*}
So if we can establish that $\Lambda\le 0$ for a suitable function $h$, then indeed $\Delta(a,b)\le 0$, and the supermartingale property follows.
\begin{rema}
Notice that the method of proof, presented here, could possibly work for $N>3$ as well; that is, if one can find a function $h(x_1,\dots,x_{N-1})$, which is positive and decreasing in each of its arguments, and $h(\XX'(t))$ is a supermartingale provided $\max\XX'(t)<\eps$ for some $\eps>0$. Unfortunately, however, we were not able to find such a function.
\end{rema}
Set
\begin{align}\label{eqfdefined}
h(x,y)=- 2\log \left(\max\left\{x,\frac{y}2\right\}\right)\ge 0;
\end{align}
it is easy to check $h$ is indeed monotone in each of its arguments as long as $x,y\in(0,1]$.
Let us now compute the integrals in the expression for $\Lambda$. We have
\begin{align*}
\Lambda= \begin{cases}
3(a-b)\ln 2-3a+2b
,&\text{ if }  a\le \frac b3; \\
(a+b)\ln(a+b)-(a+b)\ln a +(a-5b)\ln 2+b
,&\text{ if } \frac b3<a\le \frac b2; \\
(a+b)\ln(a+b) +(2a-4b)\ln b+3(b-a)\ln a +(b-5a)\ln 2+b
,&\text{ if } \frac b2<a \le \frac{2b}3; \\
(a+b)\ln(a+b)+(2a-4b)\ln b+(5b-7a)\ln a-(a+b)\ln2
\\
\quad +3(b-a)+(4a-2b)\ln(2a-b)
,&\text{ if } \frac{2b}3<a \le b.
\end{cases}
\end{align*}

It turns out that $h(\XX'(t))$ indeed has a non-positive drift, provided $0<a\le b\le \eps$, as is shown by the following
\begin{lemma}\label{lemDelta}
$\Lambda\le 0$ for $a,b\in(0,1/2]$.
\end{lemma}
\begin{proof}
Substitute $a=b\nu$ in the expression for $\Lambda$. Then for $\nu\le 1/3$ we easily obtain
$
\Lambda=-b\left[3\nu(1-\ln 2) + \ln 8 -2\right]\le 0.
$

For $1/3<\nu\le 1/2$ we have $2\Lambda=-b C_1(\nu)\le 0$ where
$$
C_1(\nu)=(1+\nu) \ln\frac\nu{1+\nu} +(5-\nu)\ln 2-1>0,
$$
since  $\frac{\partial^2 C_1(\nu)}{\partial^2 \nu}=-\frac1{\nu^2(1+\nu)}<0$ and hence $\min_{1/3\le\nu\le 1/2}C_1(\nu)$  is achieved  at one of the endpoints $\nu=1/3$ or $\nu=1/2$; the values there are $C_1(1/3)=\ln (4)-1>0$ and $C(1/2)=\frac12 \ln\left(\frac{512}{27}\right)  -1>0$  respectively.

For $1/2<\nu \le 2/3$ we have
$\Lambda=-b C_2(\nu)\le 0$ where
$$
C_2(\nu)=-(1+\nu)\ln(1+\nu)+(3\nu-3)\ln\nu-1+(5\nu-1)\ln 2>0,
$$
since $\frac{\partial^2 C_2(\nu)}{\partial^2 \nu}=\frac{2\nu^2+6\nu+3}{\nu^2(1+\nu)}>0$ and 
$\left.\frac{\partial C_2(\nu)}{\partial \nu}\right|_{\nu=2/3}=\ln\left(\frac{256}{45}\right)-\frac 52<0$
implies that $\frac{\partial C_2(\nu)}{\partial \nu}<0$ for all $\nu\in[1/2,2/3]$ and hence $\min_{1/2\le \nu\le 2/3} C_2(\nu)=C_2(2/3)=\frac13\ln\left(\frac{104976}{3125}\right)-1>0.$

Finally, for $2/3<\nu \le 1$ we have $\Lambda=-bC_3(\nu)\le 0$,
where 
$$
C_3(\nu)=
\nu
\log
\frac{2\nu^7}{(2\nu-1)^4(\nu+1)}
+\log
\frac{2  (2\nu-1)^2}{\nu^5(\nu+1)}+3(\nu-1)
>0
$$
since 
$$
\frac{\rmd^2 C_3(\nu)}{\rmd \nu^2}=
\frac{(2\nu+5)(2\nu^2-1)}{(2\nu-1)\nu^2(\nu+1))}
$$
changes its sign from $-$ to $+$ at $1/\sqrt 2\in(2/3,1)$ and therefore
$\frac{\partial C_3(\nu)}{\partial \nu}$ achieves its maximum at the endpoints of the interval; thus
$$
\max_{2/3\le\nu\le 1} \frac{\partial C_3(\nu)}{\partial \nu}=
\max_{\nu=2/3,1}\frac{\partial C_3(\nu)}{\partial \nu}
=\max\left\{-\frac 52+\ln\left(\frac{256}{45}\right),0\right\}=0
$$
Therefore, $C_3(\nu)$ is decreasing and hence 
$\min_{2/3\le\nu\le 1} C_3(\nu)=C_3(1)=0$.
\end{proof}

\begin{proof}[Proof of Theorem~\ref{thmnotzero}]
We will show that $\P(\XX'(t)\to 0)=0$, which will imply by Theorem~\ref{thmpgreater}(a) that $\P(\XX'(t)\to 1)=1$; we shall do it by showing that
\begin{align}\label{eqbyreferee}
0\le \P(\XX'(t)\to 0)\le \P\left(\left\{\liminf_{t\to\infty} x_{(1)}(t)=0 
\right\}\bigcap\left\{\limsup_{t\to\infty} x_{(2)}(t)<\eps
\right\}\right)=0.
\end{align}
Indeed, fix some $\eps\in(0,1/2)$ and let $\tau_0=0$. For $\ell=1,2,\dots,$ define the sequence of stopping times
\begin{align*}
\eta_\ell&=\inf\{t>\tau_{\ell-1}:\ x_{(2)}(t)\le \eps\},\\
\tau_\ell&=\inf\{t>\eta_{\ell}:\ x_{(2)}(t)> \eps\},
\end{align*}
so that $\tau_0<\eta_1<\tau_1<\eta_2<\tau_2<\dots$
with the conventions that if one of the stopping times is infinite, so is the rest of them. Define also $\ell^*=\inf\{\ell\ge 1:\ \tau_\ell=+\infty\}$. 

If $\ell^*=\infty$, that is, $\tau_\ell<\infty$ for all $\ell$, we immediately have $\limsup_{t\to\infty} x_{(2)}(t)\ge \eps$ and we are done; so assume that for some $\ell^*\ge 1$ we have $\tau_{\ell^*-1}<\infty=\tau_{\ell^*}$. If $\eta_{\ell^*}=\infty$, then $x_{(2)}(t)>\eps$ for all $t\ge \tau_{\ell^*-1}$ and thus again $\limsup_{t\to\infty} x_{(2)}(t)\ge \eps$; hence $\ell^*<\infty$ and $\eta_{\ell^*}<\infty$ on the event 
$\left\{\limsup_{t\to\infty} x_{(2)}(t)<\eps
\right\}$.

On the other hand, as long as $\eta_\ell<\infty$, we can define 
$$
\xi^{(\ell)}_s=h\left(\XX'(\eta_\ell+s)\right), \quad s\ge 0
\qquad\text{and}\ \tilde\xi^{(\ell)}_s=\xi^{(\ell)}_{\min\{s,\tau_\ell-\eta_\ell\}}.
$$
where $h$ is given by~\eqref{eqfdefined}. 

By Lemmas~\ref{lemsupermp} and~\ref{lemDelta} we have
\begin{align*}
\E(\left[ h(\XX'(t+1))-h(\XX'(t))\right] 1_{x_{(2)}(t)\le\eps} \|\F_t )\le 0,
\end{align*}
hence $\tilde\xi^{(\ell)}_s$, the process $\xi^{(\ell)}$ stopped at the time when $x_{(2)}$ exceeds $\eps$, is a non-negative supermartingale, hence it must converge to a finite value. In case $\tau_\ell=+\infty$ this means  that
$\liminf_{t\to\infty} x_{(1)}(t)>0$ since the function $h(a,b)$ goes to infinity when $a\downarrow 0$. Thus we have established~\eqref{eqbyreferee}.
\end{proof}

\section{ Appendix: The calculations for the proof of Lemma~\ref{lemMod}.}
Observe that all expressions for ${\bf A}_j$ are fractions of the polynomials in $(a,f,p,M)$; moreover, their denominators 
\begin{align*}
\begin{array}{ll}
3M(M-1) &\text{ (for ${\bf A}_1$)}, \\
3M(M-1)(M+1-2p)^3 &\text{ (for ${\bf A}_2$ and ${\bf A}_4$)},\\
12M(M-1)(M+1-2p)^3 p^3 &\text{ (for ${\bf A}_3$ and ${\bf A}_5$)}
\end{array}
\end{align*} 
are always positive. Throughout the rest of the proof let $\nn(w)$ denote the numerator of such a fraction $w$.

\subsubsection*{Case 1: ${\bf I}_1\le 0$}
Observe that
\begin{align*}
\nn({\bf A}_1)&=-2M^2-3M\mu+2 M+1+[3M\mu-1]Mp
\end{align*}
and the term in the square brackets is positive as $M\mu\ge 1$, so the maximum of $\nn({\bf A}_1)$ is achieved at the highest possible value of $p$. However, in this case we have $p\le p_1$, hence
\begin{align*}
\nn({\bf A}_1) 1_{X_1\le 0}\le 
\nn({\bf A}_1)|_{p=p_1}&=
-\frac{s_1}{ 2\mu }
\end{align*}
where
$$
s_1=
(M^2-2)\mu+(1-6\mu)(1-\mu)M+1=
\begin{cases}
3(2\mu-1)^2,\text{ if }M=2;\\
4\mu^2+1/2+14 (\mu-1/2)^2 ,\text{ if }M=3;\\
(M-3)[(M-4)\mu+6\mu^2+1]+s_1|_{M=3}, \text{ if }M\ge 4
\end{cases}
$$
Hence $s_1\ge 0$ for  $M=2,3,\dots$ and thus ${\bf I}_1\le 0$.

\subsubsection*{Case 2: ${\bf I}_2\le 0$}
Here
$$
\nn({\bf A}_2)=-4 \left[a(M -p+1)-M\mu p \right]^2   s_2
$$
where 
\begin{align*}
s_2&= M^3\mu p-4 M^2\mu p^2-M^3 a+2 M^2 a p+5 M^2\mu p+2 M a p^2-3 M^2\mu-6 M a*p+4 M\mu p
\\ & +3 M a-3 M\mu-2 a p+2 a,
\end{align*}
and we need to show that $s_2\ge 0$.

Assume first $M=2$. Then (using the fact that $\mu=(1+a)/2$)
\begin{align*}
X_2\ge 0 \Longleftrightarrow p\ge \frac{3a+3}{4a+2}\ge 1
\end{align*}
which is impossible; so from now on $M\ge 3$.

To establish ${\bf I}_2\le 0$, it will suffice to demonstrate that
$$
s_3:=2 M  s_2-2 M^3(M\mu+a) X_2\ge 0
$$ 
as ${\bf I}_2$ has a factor $1_{X_2\ge 0}$, and 
$s_2  1_{X_2\ge 0} \ge \frac{s_3}{2M}  1_{X_2\ge 0}$.
Substituting 
$$
p=\left[ \frac12+\frac 1{2M}\right]
+	\left[ \frac12-\frac 1{2M}\right]w
$$
where $w\in[0,1)$ corresponding to the condition~\eqref{pcond}, we get
\begin{align*}
s_3&=
M \left(-2 {M}^2\mu {w}^2+{M}^2\mu w+4 M\mu {w}^2+{M}^{3
}-3 {M}^2\mu-M\mu w-2 \mu {w}^2+{M}^2-M\mu+2 \mu \right) 
\\ &
- a\cdot (M-1) \left[
M \left(\left(M-1 \right)^2- \left(w-2 \right)^2 \right) +
 (1-w)  \left(M^2-w-1 \right) 
 \right] 
\end{align*}
The expression in the square brakets is non-negative for $M\ge 3$, so the minimum of $s_3$ is achieved when $a=1$; i.e.
\begin{align*}
s_3&\ge s_3|_{a=1}=
-2M^3\mu w^2+M^3\mu w+4 {M}^2\mu {w}^2-3 {M}^3\mu+{M}^3w-{M}^2\mu w+{M}^2{w}^2
\\ &
-2M\mu w^2+3M^3
-M^2\mu-5 M^2 w-2 M w^2+2 M^2+2 M\mu+4 Mw+w^2-2M-1=:s_4
\end{align*}
But

\begin{align*}
\frac{\partial s_4}{\partial\mu}
=-M \left((3-w)M^2+(1+w)M-2+2(M-1)^2 w^2\right) <0
\end{align*}
so
\begin{align*}
s_4 \ge s_4|_{\mu=1}&= 
(1-w)(M-1)(wM(2 M -3)+M+w+1)\ge 0.
\end{align*}

\subsubsection*{Case 3: ${\bf I}_3\le 0$}
Here
$$
\nn({\bf A}_3)= - (M+1) (1-a)   s_5
$$
and it suffices to show that $s_5\ge 0$.
If $M=2$, then $\mu=(a+1)/2$ and $p\ge 3/4$, so
\begin{align*}
s_5=3(3-2p)
&\left[(1-a)^2(8p-5)+(32(1-a)^2+144a)(1-p)^4\right.
\\ &
\left. +12(1-p)^2 (4p+a(4ap+10p-3))\right]\ge 0.
\end{align*}
For $M\ge 3$, let $M=3+\delta$, $\delta=0,1,\dots$. Then
$
s_5=\sum_{i=0}^5 e_{i+1} \delta^i
$
where we will show that all $e_i\ge 0$. Indeed, we have
\begin{align*}
e_1&=
-432a\mu p^5-1296\mu^2 p^5+288a^2 p^4+2736a
\mu p^4-144a p^5+432 p^4\mu^2-432\mu p^5-
1632 p^3a^2
 \\ &
-2880a\mu p^3+1200a p^4+4320\mu^{2
} p^3+2736\mu p^4+2624 p^2a^2-1152a\mu p^2-
3744a p^3-1728\mu^2 p^2
 \\ &
-2880\mu p^3+288 p^4-1536pa^2+4928 p^2a-1152 p^2\mu-1632 p^3+768a^2-3072ap+2624 p^2
\\ &
+1536a-1536p+768
\\
e_2&=-288a\mu p^5-1296\mu^2 p^5+168a^2 p^4+2208a
\mu p^4-48a p^5-360 p^4\mu^2-288\mu p^5-1160
 p^3a^2
  \\ &
-1392a\mu p^3+600a p^4+6984\mu^2 p^3+2208\mu p^4+1760 p^2a^2-3840a\mu p^2-2264
a p^3-2016\mu^2 p^2
 \\ &
-1392\mu p^3+168 p^4 -576pa^2+2720 p^2a-3840 p^2\mu-1160 p^3+768a^2-1152ap+1760 p^2
\\ &
+1536a-576p+768
\\
e_3&=
-48a\mu p^5-432\mu^2 p^5+24a^2 p^4+576a\mu
 p^4-600 p^4\mu^2-48\mu p^5-268 p^3a^2+
216a\mu p^3+72a p^4
  \\ & 
+4404\mu^2 p^3+576\mu p^4+324 p^2a^2-3240a\mu p^2-412a p^3-876\mu^2 p^2+216\mu p^3+24 p^4+336pa^2
 \\ &
+180 p^2a-
3240 p^2\mu-268 p^3+288a^2+672ap+324 p^2+576a
+336p+288
\\
e_4&=
-48\mu^2 p^5+48a\mu p^4-216p^4\mu^2-20p^3 a^2+192a\mu p^3+1356\mu^2 p^3+48\mu p^4-4 p^2 a^2-1164 a\mu p^2
 \\ &
-20ap^3-168\mu^2 p^2+192\mu p^3+228pa^2-112p^2 a-1164 p^2\mu -20p^3+48a^2+456ap-4p^2
\\ &
+96a+228p+48
\end{align*}
\begin{align*}
e_5&=-24  p^4\mu^2+24 a\mu  p^3+204 \mu^2 p^3
-4p^2a^2-192 a\mu  p^2-12 \mu^2 p^2+24 \mu  p^3+45 pa^2-16  p^2a
   \\ &
-192  p^2\mu+3 a^2+90 ap-4  p^2+6a+45p+3
\\
e_6&=360 p (a+1-2\mu p)^2\ge 0.
\end{align*}
The fact that $e_6\ge 0$ is trivial; we will prove separately that $e_1,\dots,e_5\ge 0$ below. In what follows, we substitute $p=\frac{1+\nu}2$, where $\nu\in(0,1)$.
\subsubsection*{Proof that $e_1\ge 0$}
We have
\begin{align*}
\frac{\partial^2 e_1}{\partial a^2}&=4  [9\nu^4-66\nu^3+76\nu^2+2\nu+235] >0 ,
\end{align*}
hence $e_1$ achieves its minimum at 
$$
a_{cr}=\frac {9\nu^5-105\nu^4+426\nu^3-46\*nu^2+397\nu-1669+9\mu(1+\nu)^2(3\nu^3-29\nu^2+13\nu+109)} {8[9\nu^4-66\nu^3+76\nu^2+2\nu+235]}
$$
which solves $\frac{\partial e_1}{\partial a}=0$. Note that it is possible that $a_{cr}\not\in [0,1]$.
However, in any case,
\begin{align*}
e_1\ge e_1|_{a=a_{cr}}=\frac1{32} \cdot \frac { 3(1+\nu)^2 c_1 }{9\nu^4-66\nu^3+76\nu^2+2\nu+235}
\end{align*}
so it will suffice to show that 
\begin{align*}
c_1&=16(1-\nu)^2 c_{1a}+3(1-\mu)c_{1b},\quad\text{where}
\\
c_{1a} &=-27\nu^6+144\nu^5-102\nu^4+1620\nu^3-9883\nu^2+12484\nu+1732
\\
c_{1b}&=81\mu\nu^8-108\mu\nu^7+135\nu^8-1260\mu\nu^6-828\nu^7-12276\mu\nu^5+276\nu^6+84774\mu\nu^4
\\ &
-4404\nu^5-157140\mu\nu^3+69170\nu^4+152628\mu\nu^2-198372\nu^3-156108\mu\nu+182084\nu^2
\\ &
+27969\mu-60588\nu+73967
\end{align*}
is positive. We have
$$
c_{1a}=3\nu^3(540-9\nu^3+48\nu^2-34\nu)+\nu(12484-9883\nu)+1732>0.
$$
Similarly,
\begin{align*}
c_{1b}&=61440(1-\mu)+(1-\nu)[c_{1b1}+c_{1b2}\mu]
\end{align*}
where
\begin{align*}
c_{1b1}&=(-135\nu^7+693\nu^6+417\nu^5+4821\nu^4)-64349\nu^3+134023\nu^2-48061\nu+12527
\\ &
\ge -64349\nu^3+134023\nu^2-48061\nu+12527\ge 1000  (-67\nu^3+134\nu^2-67\nu+12)
\\ &
=\frac{1000}{27}\left[ 56+ 67(4-3\nu) (1-3\nu)^2\right]>0
\end{align*}
and
\begin{align*}
c_{1b2}&=(-81\nu^7+27\nu^6+1287\nu^5+13563\nu^4)-71211\nu^3+85929\nu^2-66699\nu+894
\\ &
\ge -71211\nu^3 +85929\nu^2-66699\nu+89409> 80000  (-\nu^3+\nu^2-\nu+1)\ge 0.
\end{align*}
So, $c_{1b1},c_{1b2}>0$ $\Longrightarrow$ $c_{1b}>0$ and since $c_{1a}>0$ we have $c_1\ge 0$ and thus $e_1\ge 0$.

\subsubsection*{Proof that $e_2\ge 0$}
We have
\begin{align*}
\frac{\partial^2 e_2}{\partial a^2}&=21\nu^4-206\nu^3+136\nu^2+398\nu+1571>0
\end{align*}
so, similarly to the previous case,
$$
e_2\ge e_2|_{a=a_{cr}}=\frac {3(1+\nu^2)[582912(1-\mu)^2+(1-\nu)c_2]}
{8[21\nu^4-206\nu^3+136\nu^2+398\nu+1571]}
$$
where
$$
a_{cr}=\frac {3 \nu^5-60 \nu^4+296 \nu^3-82 \nu^2-155 \nu-2786+
 \left(18 \nu^3-222 \nu^2-150 \nu+2010 \right)  \left(1+
\nu \right)^2\mu}{2[21\nu^4-206\nu^3+136\nu^2+398\nu+1571]}
$$
solves $\frac{\partial e_2}{\partial a}=0$ and
\begin{align*}
c_2 &=3 \nu^{7}-123 \nu^{6}+1330 \nu^5-1918 \nu^4-28897 {
\nu}^3+65177 \nu^2+93100 \nu+120544
\\ &
+ \left(36 \nu^{7}-624 \nu^{6}+348 \nu^5+25616 \nu^4-7332 \nu^3-272368 \nu^2-134556 \nu+688784 \right) 
 \\ &
 + \left(108 \nu^{7}-72 \nu^{6}-4848 \nu^5-35916 \nu^4+247548 \nu^3-252720 \nu^2+144456 \nu-647676 \right) {\mu}^2
\end{align*}
Now, 
\begin{align*}
\frac{\partial^2 c_2}{\partial \mu^2}&=
\nu^4(216\nu^3-144\nu^2-9696\nu-71832)+\nu^3(495096\nu-505440)+(288912\nu-1295352)<0
\end{align*}
hence the minimum of $c_2$ w.r.t.\ $\mu\in[0,1]$ can be achieved either at $\mu=0$ or at $\mu=1$.
At the same time
\begin{align*}
c_2|_{\mu=0}&=3\nu^7+1330\nu^5+65177\nu^2+93100\nu+(120544-123\nu^6-1918\nu^4-28897\nu^3) >0,
\\
c_2|_{\mu=1}&=(1-\nu)(
161652-147\nu^6+672\nu^5+3842\nu^4+16060\nu^3+(264652-195259\nu)\nu)\ge 0,
\end{align*}
so $c_2\ge 0$ and hence $e_2\ge 0$.

\subsubsection*{Proof that $e_3\ge 0$}
We have
\begin{align*}
\frac{\partial^2 e_3}{\partial a^2}&=3\nu^4-55\nu^3-21\nu^2+471\nu+1010>0
\end{align*}
so, similarly to the previous case,
$$
e_3\ge e_3|_{a=a_{cr}}=
\frac{3(1+\nu)^2 \left[(1-\nu)^2 c_{3a}+(1-\mu)c_{3b}\right]}{8(3\nu^4-55\nu^3-21\nu^2+471\nu+1010)}
$$
where
$$
a_{cr}=
\frac{-9\nu^4+67\nu^3+165\nu^2-579\nu-1820+3(1+\nu)^2 \mu(\nu^3-21\nu^2-63\nu+499)}{2(3\nu^4-55\nu^3-21\nu^2+471\nu+1010)}
$$
solves $\frac{\partial e_3}{\partial a}=0$ and
\begin{align*}
c_{3a} &=-3\nu^6+12\nu^5+632\nu^4+1794\nu^3-37624\nu^2+65244\nu+64877>0
\\
c_{3b} &=2(1-\nu)(-3\nu^7+12\nu^6+652\nu^5+2417\nu^4-42561\nu^3+73864\nu^2+41336\nu+91323)
\\ &
+(1-\mu) (3\nu^6(220-\nu^2+4\nu)+2\nu(1490\nu^4-22993\nu^3+39898\nu^2+890\nu+109262)+8477)\ge 0
\end{align*}
Hence $e_3\ge 0$.

\subsubsection*{Proof that $e_4\ge 0$}
We have
\begin{align*}
\frac{\partial^2 e_4}{\partial a^2}&=
209\nu+317-5\nu^3-17\nu^2>0
\end{align*}
so, similarly to the previous case,
$$
e_4\ge e_4|_{a=a_{cr}}=\frac{3(1+\nu)^2[ (1-\nu)^2 c_{4a}+4(1-\mu)c_{4b}]}
{8(209\nu+317-5\nu^3-17\nu^2)}
$$
where
$$
a_{cr}=\frac{5\nu^3+71\nu^2-329\nu-587+  6\mu(1+\nu)^2(88-10\nu-\nu^2)}
{2(209\nu+317-5\nu^3-17\nu^2)}
$$
solves $\frac{\partial e_4}{\partial a}=0$ and
\begin{align*}
c_{4a}&=8\nu^4+40\nu^3-1395\nu^2+4354\nu+4757>0
\\
c_{4b}&=4(1-\nu) (4\nu^5+21\nu^4-712\nu^3+2011\nu^2+3102\nu+3050)
\\ & 
+(1-\mu)(2\nu^6+11\nu^5-360\nu^4+912\nu^3+1705\nu^2+3655\nu+543)\ge 0.
\end{align*}
Hence $e_4\ge 0$.

\subsubsection*{Proof that $e_5\ge 0$}
We have
\begin{align*}
\frac{\partial^2 e_5}{\partial a^2}&=49+41\nu-2\nu^2>0
\end{align*}
so, similarly to the previous case,
$$
e_5\ge e_5|_{a=a_{cr}}=\frac{3(1+\nu)^2[ (1-\nu)^2 c_{5a}+(1-\mu)c_{5b}]}{2(49+41\nu-2\nu^2)}
$$
where
$$
a_{cr}=\frac{4\nu^2-37\nu-47+3\mu (15-\nu)(1+\nu)^2}{49+41\nu-2\nu^2}
$$
solves $\frac{\partial e_5}{\partial a}=0$ and
\begin{align*}
c_{5a}&=15-\nu^2+15\nu>0
\\
c_{5b}&=2(1-\nu) (14\nu^2+28\nu+19-\nu^3)
+
(1-\mu)(13\nu^3+40\nu^2+49\nu+11-\nu^4)\ge 0.
\end{align*}
Hence $e_5\ge 0$.
\vskip 5mm
As a result, $s_5\ge 0$ and thus ${\bf I}_3\le 0$.

\subsubsection*{Case 4: ${\bf I}_4\le 0$}
Here
\begin{align*}
\nn({\bf A}_4)&=-4(M\mu p-M+p-1)^2 s_6,\\
s_6 &=2 p-2+(3\mu+6p-3-4\mu p-2p^2)M+(4\mu p^2-5\mu p+3\mu-2p)M^2+(1-\mu p)M^3
\end{align*}
Then, substituting $M=2+\d$,
\begin{align*}
\frac{\partial s_6}{\partial \d}&=
[5(1-\mu)+2(1-p)(p+2+10\mu-8\mu p)]
+
[8(1-\mu)+2(1-p)(2+7\mu-4\mu p)]\d\ge 0
\end{align*}
and as a result for $\d\ge 0$ we have
$$
s_6\ge s_6|_{\d=0}=2(3-2p)[p(1-\mu)+\mu(1-3p)]\ge 0.
$$

\subsubsection*{Case 5: ${\bf I}_5\le 0$}
Here
$$
\nn({\bf A}_5)=-s_7.
$$
We need to show that $s_7\ge 0$ when $X_1\le 0$ and $X_3\ge 0$.

Since $X_1\le 0$, we have $2Mp\mu\le M+1$. 
Together with $X_3\ge 0$ this implies
$$
0\le \nn(X_3)=2 M p\mu-(M+1)-a(M+1)+2p
\le -a(M+1)+2p
$$
whence
$$
a\le \frac{2p}{M+1}.
$$
Let us show that for this $a$ we have $s_7\ge 0$; substitute  $a=b\cdot \frac{2p}{M+1}$, where $b\in[0,1]$.

First, let \boxed{$M=2$}\,, then $\mu=\frac{1+a}2$, $p\in [3/4,1)$, and
$s_7=\frac{3-2p}{27} s_8$ where
\begin{align*}
s_8 &=
512 b^3 p^8-2688 b^3 p^7+5760 b^3 p^6+3456 b^2 p^7-6912 b^3 p^5-12672 b^2 p^6+5184 b^3 p^4
\\ &
+16416 b^2 p^5+5184 b p^6-1944 b^3 p^3-11664 b^2 p^4-10368 b p^5+7776 b^2 p^3+1728 p^5-2916 b^2 p^2
\\&
+11664 b p^3-11664 b p^2-7776 p^3+4374 b p+17496 p^2-17496 p+6561
\end{align*}
Note that we can write $s_8=e_1+e_2 (1-p) +e_3 (1-p)^2$, where
\begin{align*}
128  e_1&=
(9- \nu^2-6\nu)(81-\nu^3-9\nu^2-63\nu)(\nu^3+15\nu^2+81-9\nu)
>0
\\
128  e_2&=3(9-\nu^2)(\nu^6+21\nu^5+168\nu^4+666\nu^3+81\nu^2+81\nu+486)
>0
\\
64(e_1+e_3)&=
[2\nu^8+33\nu^7+234\nu^6+783\nu^5]+
[-648\nu^4-6561\nu^3+30618\nu^2-28431\nu+13122]
\\ &
 \ge -648\nu^4-6561\nu^3+30618\nu^2-28431\nu+13122
\\  & 
  \ge -1000\nu^4-7000\nu^3+24000\nu^2-29000\nu+13000
\\ &
= 1000(1-\nu)(5 +8(1-\nu)^2 +\nu^3) \ge 0.
\end{align*}
with  $p=\frac{3+\nu}4$, $\nu\in [0,1]$. Consequently, since $(1-p)^2<1$ and $e_1>0$,
$$
s_8=e_1+e_2 (1-p) +e_3 (1-p)^2\ge e_2 (1-p) +(e_1+e_3) (1-p)^2
\ge 0
$$
and thus $s_7\ge 0$ as required.
\vskip 5mm
For $\boxed{M\ge 3}$\,, set $M=3+\d$, $\d\ge 0$. Then
$$
s_7=\sum_{i=0}^{9} e_{i+1}\d^i
$$
where
\begin{align*}
e_1&=196608+ (98304 b-393216) p+ (-49152 b^2-442368 \mu^2-196608 b-737280 \mu+589824) p^2
\\ &
+
 (-24576 b^3+221184 b\mu^2+98304 b^2+1990656 \mu^2+294912 b+663552 \mu-540672) p^3
\\ &
+ (49152 b^3+73728 b^2\mu-552960 b\mu^2-331776 \mu^3-147456 b^2-2322432 \mu^2-270336 b
\\ & \ \ 
+110592 \mu+233472) p^4
\\ &
+ (-83968 b^3+184320 b^2\mu-55296 b\mu^2+774144 \mu^3+135168 b^2+1050624 \mu^2+116736 b
\\ & \ \
-239616 \mu-36864) p^5
\\ &
+ (52224 b^3-175104 b^2\mu+165888 b\mu^2-331776 \mu^3-58368 b^2-165888 \mu^2-18432 b+55296 \mu) p^6
 \\ & + (-9216 b^3+27648 b^2
\mu+9216 b^2) p^7
\end{align*}
\begin{align*}
e_2&=393216+ (172032 b-540672) p+ (-73728 b^2-958464 \mu^2-221184 b-2088960 \mu+835584) p^2
\\ &
+(-30720 b^3+423936 b\mu^2+86016 b^2+4589568 \mu^2+344064 b+1898496 \mu-823296) p^3
\\ &
+ (30720b^3+282624 b^2\mu-1308672 b\mu^2-663552 \mu^3-135168b^2-5031936\mu^2-344064 b
\\ & \ \
-18432 \mu+344064)p^4
\\ & 
 + (-77312 b^3+181248 b^2\mu+4608 b\mu^2+1658880 \mu^3+138240 b^2+2068992 \mu^2+142848b
\\ & \ \ 
 -360960 \mu-49152) p^5
\\ & 
 + (50176 b^3-228864 b^2\mu+290304 b\mu^2-691200 \mu^3-56832 b^2-290304\mu^2-19968 b+78336 \mu) p^6
\\ &
+ (-7680 b^3+32256 b^2\mu+7680 b^2) p^7
\end{align*}
\begin{align*}
e_3&=344064+ (129024 b-208896) p+ (-46080 b^2-906240 \mu^2-49152 b-2558976 \mu+430080) p^2
\\ &
+ (-15360 b^3+347136 b\mu^2+3072b^2+4718592\mu^2+129024 b+2217984 \mu-522240) p^3
\\ &
+ (-6144 b^3+334848 b^2\mu-1337856 b\mu^2-566784 \mu^3-30720 b^2-4778496 \mu^2-175104 b
\\ & \ \ 
-198144 \mu+210432) p^4
\\ &
+ (-24448 b^3+42240 b^2\mu+100992 b\mu^2+1543680 \mu^3+52992 b^2+1744512 \mu^2+69504 b\\ & \ \
-223872 \mu-26112) p^5
\\ &
+ (17856 b^3-118464 b^2\mu+210816 b\mu^2-615168 \mu^3-20544 b^2-210816 \mu^2-8064 b+44160 \mu) p^6
\\ &
+ (-2112 b^3+14016 b^2\mu+2112 b^2) p^7
\end{align*}
\begin{align*}
e_4&=172032+ (53760 b+64512) p+ (-15360 b^2-488448 \mu^2+44544 b-1790976 \mu+64512) p^2
\\&
+ (-3840 b^3+157440 b\mu^2-23040 b^2+2843904 \mu^2+1416960 \mu-176640) p^3
\\&
+ (-9984 b^3+193536 b^2\mu-772608 b\mu^2-268032 \mu^3+7680 b^2-2598912 \mu^2-44544 b
\\ & \ \
-170496 \mu+68352) p^4
\\ &
+ (93024 b\mu^2-13632 b^2\mu+32\mu(25448\mu^2+25515\mu-2283)-32b(77b^2-282b-525)-6912) p^5
\\ &
+ (2784 b^3-30336 b^2\mu+81312 b\mu^2-303168 \mu^3-3264 b^2-81312 \mu^2-1440 b+12384 \mu) p^6
\\ &
+ (-192 b^3+2688 b^2\mu+192 b^2) p^7
\end{align*}
\begin{align*}
e_5&=53760+ (13440 b+91392) p+ (-2880 b^2-164160 \mu^2+34560 b-792768 \mu-26880) p^2
\\ & 
 + (-480 b^3+42720 b\mu^2-11520 b^2+1109088 \mu^2-13440 b+548640 \mu-33600) p^3
\\ & 
 + (62496 b^2\mu-275952 b\mu^2-885744 \mu^2-67536 \mu-75792 \mu^3-96b(34 b^2-50b+59)+12432) p^4
\\ & 
 + (160 b^3-8544 b^2\mu+38928 b\mu^2+266192 \mu^3+576 b^2+229104\mu^2+2016 b-13200 \mu-912) p^5
\\ & 
 + (160 b^3-3840 b^2\mu+17568 b\mu^2-89344 \mu^3-192 b^2-17568 \mu^2-96 b+1728 \mu) p^6+192 b^2\mu p^7
\end{align*}
\begin{align*}
e_6&=10752+ (2016 b+38976) p+ (-288 b^2-35232 \mu^2+10848 b-230880 \mu-14784) p^2
\\ &
+ (-24 b^3+6936 b\mu^2-2544 b^2+290664 \mu^2-4032 b+132888 \mu-3408) p^3
\\ &
+ (11568 b^2\mu-456 b^3-62472 b\mu^2-12816 \mu^3+816 b^2-193752 \mu^2-288 b-14328 \mu+1200) p^4
\\ &
+ (32 b^3-1536 b^2\mu+8688 b\mu^2+55168 \mu^3+38544 \mu^2+96 b-1248 \mu-48) p^5
\\ & 
 + (-192 b^2\mu+2016 b\mu^2-15744 \mu^3-2016 \mu^2+96 \mu) p^6
\end{align*}
\begin{align*}
e_7&=1344+ (168 b+9072) p+ (-12 b^2-4716 \mu^2+1824 b-44340 \mu-3024) p^2
\\ &
+ (624 b\mu^2-276 b^2+51252 \mu^2-504 b+19764 \mu-144) p^3
 \\&
 + (-24 b^3+1152 b^2\mu-8760 b\mu^2-1200 \mu^3+48 b^2-26568 \mu^2-1584 \mu+48) p^4
\\ &
+ (-96 b^2\mu+1008 b\mu^2+7072 \mu^3+3600 \mu^2-48 \mu) p^5+ (96 b\mu^2-1536 \mu^3-96 \mu^2) p^6
\end{align*}
\begin{align*}
e_8&=96+ (6 b+1236) p+ (-360 \mu^2+162 b-5424 \mu-300) p^2
\\ &
+ (24 b\mu^2-12 b^2+5868 \mu^2-24 b+1656 \mu) p^3+ (48 b^2\mu-696 b\mu^2-48 \mu^3-2088 \mu^2-72 \mu) p^4
\\&
+ (48 b\mu^2+512 \mu^3+144 \mu^2) p^5-64 \mu^3 p^6
\end{align*}
and the expressions for $e_9$ and $e_{10}$ are given a little bit further.

First, we will show that  $e_i\ge 0$, $i=1,\dots,8$.

\subsubsection*{Proof that $e_1,\dots, e_8> 0$}
It turns out that it is easiest is to use a computer-assisted proof in this case; to this end we developed the method which we call a {\em Box method}; 
it may have been described by other authors, but since we do not have the reference to the right source, we give its description below.

First of all, we substitute
$$
p=\frac{1+x_1}2,\ b=x_2, \mu=x_3; \quad x_i\in[0,1],\ i=1,2,3.
$$
Let ${\rm m}=\min_{a_i\le x_i\le b_i,i=1,2,3} f(x_1,x_2,x_3)$
where 
$$
f(x_1,x_2,x_3)=f_+(x_1,x_2,x_3)-f_-(x_1,x_2,x_3)
$$
and $f_+$ and $f_-$ are polynomials with non-negative coefficients. We want to show that   ${\rm m}>0$.

Let 
\begin{align*}
G_{f;M}&
=\min_{i_1,i_2,i_3=0,\dots,M-1} \left[
f_+\left(
\frac{i_1}{M}  ,
\frac{i_2}{M}  ,
\frac{i_3}{M}  
\right)
-
f_-\left(
\frac{i_1+1}{M}  ,
\frac{i_2+1}{M}  ,
\frac{i_3+1}{M}  
\right)\right].
\end{align*}
Since 
$$
{\rm m} \ge 
G_{f;M}\to
{\rm m}
$$
as $M\to\infty$, we conclude that ${\rm m}>0$ if and only if $G_{f,M}\ge 0$ for some $M\ge 1$. Checking that $G_{f,M}\ge 0$ can be quite tedious and time-consuming for large $M$, however, this could be easily accomplished with the help of a computer; please note, that the results are still {\em completely rigorous}, unlike e.g.\ simulations.

The results of application of this method to $e_1,\dots,e_8$ are presented in the following table:
\begin{center}
\begin{tabular}{ c c c c }
 $G_{e_1,2000}>825$, & $G_{e_2,500}>25$, & $G_{e_3,400}>1860$, & $G_{e_4,300}>2397$, \\ 
$G_{e_5,200}>672$, & $G_{e_6,200}>148$, & $G_{e_7,200}>5$, & $G_{e_8,400}>3$.
\end{tabular}
\end{center}
Consequently, $e_j>0$ for all $j=1,\dots,8$.

\subsubsection*{Proof that $e_9\ge 0$ and $e_{10}\ge 0$}
The Box method of the previous section would not work for $e_9$ and $e_{10}$, since these functions do touch zero in the required area, and hence the minimum is, in fact, $0$. Therefore, we have to handle these two cases analytically.

We have
\begin{align*}
e_9&=
4 p^2\mu(4\mu^2 p^3-18\mu p^2+99\mu p-3+15 p-96)
-12p^2+93p+3
+[6 p^2 (1-2\mu p) (2\mu p+1)]b,
\end{align*}
hence, the minimum is achieved either at $b=0$ or $b=1$.

For $\mu<1/(2p)$ we have  $e_9\ge e_{9a}$, where
\begin{align*}
e_{9a}&=e_9|_{b=0}=2 s^3 p^2-18 p^2 s^2+30 s p^2+99 s^2 p-12 p^2-192ps-3s^2+93p+3
\\
&=2p^2+(1-s)[6(1-p)+(1-s)(99p+2p^2s-14p^2-3)]\ge 0
\end{align*}
where $s=2p\mu\in[0,1]$.

In case $\mu \ge 1/(2p)$ we have $e_9\ge e_{9b}$, where
\begin{align*}
e_{9b}&=e_9|_{b=1}=
16p^5 s^3-24p^4 s^3-72p^4s^2+12p^3 s^3+468 p^3 s^2-2s^3 p^2-24p^3s-426p^2 s^2
\\ &
+24s p^2+111s^2 p+2p^2-18ps-3s^2+6s
\end{align*}
where $\mu=\frac1{2p}+s\left(1-\frac1{2p}\right)$, $s\in[0,1]$.
Now,
\begin{align*}
\frac{\partial^2}{\partial s^2}e_{9b}&=
6(2p-1)^2 (14+(2p-1)(2p^2 s-3p+15))\ge 0
\end{align*}
so the minimum of $e_{9b}$ w.r.t. $s$ is achieved where 
$\frac{\partial}{\partial s}e_{9b}=0$, i.e.
$$
s_{cr}={\frac {6p^2-33p+1+R}{2p^2 (2p-1) }}
,\quad\text{where}\quad
R=\sqrt {44p^4-400p^3+1105p^2-66p+1}
$$
and equals 
\begin{align*}
&\frac{3996p^5-284p^6-19956p^4+37329p^3-3291p^2+99p-1+(400p^3-44p^4-1105p^2+66p-1)R}{2p^4}
\\
& \ge 22120.5-1576\sqrt{197}=0.285896\dot>0
\end{align*}  
for $p\ge 1/2$.
\\[5mm]
Finally, trivially, we have $e_{10}=3p(2\mu p-1)^2\ge 0.$
Consequently, $s_7\ge 0$ and ${\bf I}_5\le 0$.

Combining this with the previously established inequalities 
${\bf I}_j\le 0$, $j=1,2,3,4$, we complete the proof Lemma~\ref{lemMod}.

\section*{Acknowledgement}
We would like to thank the anonymous referees for very careful reading our manuscript  and for giving many useful suggestions and corrections.

\end{document}